\documentclass[11pt]{article}

\def\UseSection{
        \numberwithin{equation}{section}
        \newtheorem{theorem}    {Theorem}[section]
        \DefineTheorems 
}
\usepackage[textwidth=500pt,textheight=650pt,centering]{geometry} 

\usepackage{amsfonts}
\usepackage{amsmath,amssymb,amsthm}
\usepackage{appendix}
\usepackage{bbm} 
\usepackage{amsbsy}
\usepackage{enumerate}
\usepackage{cite}
\usepackage{wasysym}
\usepackage{comment}

\usepackage[bookmarks, colorlinks, breaklinks,
pdftitle={},pdfauthor={}]{hyperref}

\hypersetup{
  linkcolor=black,
  citecolor=black,
  filecolor=black,
  urlcolor=black}

\usepackage{verbatim}
\usepackage{tikz}
\usepackage{stmaryrd}
\usepackage{color}

\newcommand{\black}{\black}

\usepackage[font=small,labelfont=bf]{caption}

\numberwithin{equation}{section}

\usetikzlibrary{calc,decorations.markings}
\usetikzlibrary{decorations.pathmorphing}
\usetikzlibrary{decorations.pathreplacing}
\usetikzlibrary{patterns}
\usetikzlibrary{arrows}

\newcommand{\bb}[1]{\mathbb{#1}}

\newcommand{\floor}[1]{\lfloor #1 \rfloor}
\newcommand{\ceil}[1]{\lceil #1\rceil}

\newcommand{\D}[0]{{\, \operatorname{d}}}

\newcommand{\blank}[1]{}

\newcommand{\E}{\bb E}

\newcommand{\Z}{\bb Z}
\newcommand{\C}{\bb C}

\newcommand{\N}{\bb N}

\renewcommand{\P}{\bb P}
\newcommand{\Q}{\bb Q}

\newcommand{\Bcal}   {\mathcal{B}}
\newcommand{\Ccal}   {\mathcal{C}}

\newcommand{\Gcal}   {\mathcal{G}}

\newcommand{\Lcal}   {\mathcal{L}}
\newcommand{\Mcal}   {\mathcal{M}}

\newcommand{\Wcal}   {\mathcal{W}}

\newcommand{\nnb}	{\nonumber \\}
\newcommand{\bubble}{{\sf B}}

\def\DefineTheorems{
	\newtheorem{lemma}      [theorem] {Lemma}
	\newtheorem{cor}        [theorem] {Corollary}

	\newtheorem{prop}        [theorem] {Proposition}
	\theoremstyle{definition}
	\newtheorem{defn}       [theorem] {Definition}
	\newtheorem{rk}       [theorem] {Remark}
}

\UseSection
\title{Self-avoiding walk on the hypercube}

 \author{
    Gordon Slade\thanks{Department of Mathematics,
     University of British Columbia,
     Vancouver, BC, Canada V6T 1Z2.
     \url{https://orcid.org/0000-0001-9389-9497}, {\tt slade@math.ubc.ca}.}}

 \date{\vspace{-5ex}} 

\begin{document}
\maketitle

\begin{abstract}
We study the number $c_n^{(N)}$
of $n$-step self-avoiding walks on the $N$-dimensional hypercube, and
identify an $N$-dependent \emph{connective constant} $\mu_N$ and amplitude $A_N$
such that $c_n^{(N)}$ is
$O(\mu_N^n)$ for all $n$ and $N$, and is asymptotically
$A_N \mu_N^n$
as long as $n\le 2^{pN}$ for any fixed $p< \frac 12$.
We refer to the regime $n \ll 2^{N/2}$
as the \emph{dilute phase}.
We discuss conjectures concerning different behaviours of $c_n^{(N)}$
when $n$ reaches and exceeds $2^{N/2}$, corresponding to a critical window
and a dense phase.
In addition, we prove that the connective constant has an asymptotic
expansion to all orders in $N^{-1}$, with integer coefficients, and we compute the first
five coefficients $\mu_N  = N-1-N^{-1}-4N^{-2}-26N^{-3}+O(N^{-4})$.
The proofs are based on generating function and Tauberian methods implemented via
the lace expansion, for which an introductory account is provided.
\end{abstract}

%

\section{Introduction}
\label{sec:1}

The self-avoiding walk in a much studied model in combinatorics, probability theory,
statistical physics, and polymer chemistry \cite{Hugh95,MS93}.
Typically it has been studied on an infinite
graph such as the hypercubic lattice $\Z^d$.  More recently its critical behaviour
has been analysed on finite graphs including the complete graph \cite{Slad20,DGGNZ19}
and (for weakly self-avoiding walk)
a discrete torus in dimensions $d >4$ \cite{Slad20_wsaw,MS22,Mich22}.  Our goal here is to
investigate the critical behaviour of the self-avoiding walk on the hypercube.
We analyse its dilute phase in detail using the lace expansion, and identify the
connective constant, whose reciprocal is the critical value.
We also raise open questions about its critical window
and dense phase.

\subsection{Self-avoiding walk on the hypercube}

Let $\Q^N = \Z_2^N$ denote the $N$-dimensional hypercube.
Thus an element $x \in \Q^N$ is a binary string of length $N$.
Addition on $\Q^N$ is defined coordinate-wise modulo $2$.
The \emph{volume}
of $\Q^N$ is
$V=V(N)=2^N$.  The Hamming norm $|x|$ of $x\in \Q^N$  is the number
of coordinates of $x$ which are equal to $1$.  In particular $|x|$ is an integer between $0$ and $N$.

An $n$-step \emph{walk} on $\Q^N$ is a function $\omega :\{0,1,\ldots,n\} \to \Q^N$ with
$|\omega(i)-\omega(i-1)| =1$ for $1 \le i \le n$.
An $n$-step \emph{self-avoiding walk} on $\Q^N$ is an $n$-step walk for which
$\omega(i) \neq \omega(j)$ for all $i \neq j$.  Typically we take $\omega(0)=0$.
Let $c_n^{(N)}$ be the number of $n$-step self-avoiding walks on $\Q^N$ with $\omega(0)=0$.
For $n=0$ we set $c_0^{(N)}=1$.

For example, the $3$-step walk $00000, \, 00100, \, 01100, \, 01000$ is counted in $c_3^{(5)}$,
and for any $N\ge 1$  we have $c_0^{(N)}=1$, $c_1^{(N)}=N$, $c_2^{(N)}=N(N-1)$, $c_3^{(N)}=N(N-1)^2$,
$c_4^{(N)}= N^2(N-1)(N-2)$.
Since an $n$-step self-avoiding walk visits $n+1$ distinct vertices,
$c_n^{(N)}=0$ if $n\ge V$.  Also, $c_{V-1}^{(N)}$ is the number of Hamilton paths
on $\Q^N$ which start at $0$.
Our aim is to study the asymptotic behaviour of $c_n^{(N)}$ for large $n$ and $N$.

The \emph{susceptibility} is the generating function for the sequence $c_n^{(N)}$
(for fixed $N$), and thus is the polynomial in $z \in \C$ defined by
\begin{equation}
    \chi_N(z) = \sum_{n=0}^\infty c_n^{(N)}z^N = \sum_{n=0}^{V-1} c_n^{(N)}z^N.
\end{equation}
Motivated by the definition of the critical value for self-avoiding walk
on a finite graph proposed in \cite{Slad20},
which itself was motivated by finite-graph percolation \cite{BCHSS05a},
given any $\lambda >0$ we define the \emph{critical value} $z_N = z_N(\lambda) >0$ by
\begin{equation}
\label{e:zNdef}
    \chi_N(z_N) = \lambda V^{1/2} = \lambda 2^{N/2}.
\end{equation}
To ensure that $z_N$ is well-defined, we always assume that $\lambda V^{1/2} \ge \chi_N(0)=1$.
Then we define the \emph{connective constant} $\mu_N=\mu_N(\lambda)$ to be the reciprocal
of the critical value:
\begin{equation}
\label{e:mudef}
    \mu_N = \mu_N(\lambda) = \frac{1}{z_N(\lambda)}.
\end{equation}
The term ``constant'' is used despite the dependence of $\mu_N$ on $N$ and $\lambda$.
By definition, $z_N$ is an increasing function of $\lambda$, and $\mu_N$ is decreasing.

\subsection{Main results}

Our main results are the following five theorems, Theorems~\ref{thm:cnbd}--\ref{thm:EL-new}.
We expect that with minor additional effort it would be possible to extend our results to more
general graphs including the Hamming graph, as in \cite{BCHSS05b}.
However we prefer to restrict attention to the hypercube to develop methods in a
concrete setting.

\subsubsection{Connective constant and number of self-avoiding walks}

As a first indication that the connective constant is useful,
the following theorem shows that it provides an exponential upper bound on the number of
$n$-step self-avoiding walks, valid for all $n \in \N$.

\begin{theorem}
\label{thm:cnbd}
There exist $\lambda_0>0$ and $K>0$ (depending on $\lambda_0$) such that for all $n,N \in \N$ (with $\lambda_0 V^{1/2}\ge 1$),
\begin{equation}
    c_n^{(N)} \le K  \mu_N(\lambda_0)^n.
\end{equation}
\end{theorem}

The next theorem establishes that the connective constant truly is the exponential growth
rate of the number of $n$-step self-avoiding walks as long as $n \le V^p$ for any fixed
$p\in (0,\frac 12)$.
We regard this range of $n$ as the regime in which the self-avoiding walk does not yet
``feel'' the finite volume of the hypercube.
A more detailed error estimate is given in Theorem~\ref{thm:dilute}.

\begin{theorem}
\label{thm:dilute-short}
There exists $\lambda_0>0$ such that with $\mu_N=\mu_N(\lambda)$ defined by \eqref{e:mudef}
for any $\lambda \in (0,\lambda_0]$,
and for any choice of $p \in (0,\frac 12)$,
there exists $\epsilon_p>0$ such that
\begin{equation}
\label{e:cnmr-short}
    c_{n}^{(N)} = A_N \mu_N^n
    \left[ 1   +O(  n^{-\epsilon_p})\right]
\end{equation}
for all $n,N$ such that $n \le V^p$ (and $\lambda V^{1/2}\ge 1$).
The sequence $A_N$ is independent of $n$ (but depends on $\lambda$ and $p$).
The constant in the error term depends on $\lambda$ and $p$ but not on $n$ or $N$
as long as $n \le V^p$.
\end{theorem}

A possibly surprising feature of \eqref{e:cnmr-short} is that its
left-hand side does not depend on the choice of $\lambda$
but both $A_N$ and the exponential term $\mu_N^n$ on the right-hand side do
depend on $\lambda$.
This is not contradictory, as we will prove
in Section~\ref{sec:dzPi} (see \eqref{e:zratio})
that, for $0<\lambda' \le \lambda_1 < \lambda_2 \le \lambda_0$,
\begin{equation}
\label{e:muratio}
    \frac{\mu_N(\lambda_1)}{\mu_N(\lambda_2)} = 1 + O(V^{-1/2}),
\end{equation}
where the constant in the error term depends on $\lambda',\lambda_0$.
Thus the replacement of one fixed choice of $\lambda$ by another in $\mu_N^n$
produces a factor $[1+O(V^{-1/2})]^n$, and for $n \le V^p$ with $p<\frac 12$ this
is $1+O(nV^{-1/2})$ and hence can be absorbed by the error term $n^{-\epsilon_p}$
since when $n \le V^p$ we have
\begin{equation}
    \frac{n}{V^{1/2}} \le \frac{1}{n^{(1-2p)/(2p)}}.
\end{equation}
The next theorem gives another sense in which
the connective constant $\mu_N$ depends only weakly on $\lambda$
and the amplitude $A_N$ depends only weakly on $\lambda$ and $p$.

\begin{theorem}
\label{thm:expansion}
Let $\lambda_0>0$ be sufficiently small.
Let $m \in \N$, fix $c>0$
(independent of $N$ but possibly depending on $m$), and suppose that $z$ obeys
$\chi_N(z) \in [cN^{m}, \lambda_0 V^{1/2}]$.
Then there are integers
$a_n$ for $n \in \N$, which are universal constants that do not depend on the particular choice
of $z$, such that
\begin{equation}
\label{e:aeqn}
    z = \sum_{n=1}^m a_n N^{-n} + O(N^{-m-1})
    .
\end{equation}
The constant in the error term depends on $m,\lambda_0,c$, but does not
depend otherwise on $z$.
The first five terms are given by
\begin{equation}
\label{e:zNexpansion}
    z = \frac{1}{N} + \frac{1}{N^2} + \frac{2}{N^3} + \frac{7}{N^4}
    + \frac{39}{N^5}  + O\Big( \frac{1}{N^6} \Big)
\end{equation}
For any $\lambda \in (0,\lambda_0]$, $p \in (0,\frac 12)$, and $m \in \N$,
the amplitude $A_N$ in \eqref{e:cnmr-short} has an asymptotic expansion
\begin{equation}
    A_N = \sum_{n=1}^m a_n' N^{-n} + O(N^{-m-1})
    .
\end{equation}
with universal integer coefficients $a_n'$ (which in particular do not depend
on $p,\lambda$)
and with an error depending on $m,\lambda,p$.
The first five terms are given by
\begin{equation}
\label{e:ANexpansion}
    A_N = 1 + \frac{1}{N} + \frac{4}{N^2} + \frac{26}{N^3} + \frac{231}{N^4}
    + O\Big( \frac{1}{N^5} \Big).
\end{equation}
\end{theorem}

By Theorem~\ref{thm:expansion}, any choice of $z$ for which
$\chi_N(z) \in [cN^m, \lambda_0 V^{1/2}]$ has the same expansion
up to an error $O(N^{-m-1})$, with the error independent of
the particular choice made for $z$.
The expansion \eqref{e:aeqn} is valid simultaneously to all orders $m$ if we choose
an $N$-dependent sequence $z$ for which $\chi_N(z)$ lies eventually in all intervals
$[cN^{m}, \lambda_0 V^{1/2}]$.
In particular, \eqref{e:aeqn} holds
simultaneously for all $m$ when $z=z_N(\lambda)$ with $\lambda \in (0,\lambda_0]$,
with the coefficients $a_n$ independent of $\lambda$.
It also holds if $z$ is chosen, e.g., to satisfy $\chi_N(z)=2^{\sqrt{N}}$.
The connective constant therefore also has an asymptotic expansion in $N^{-1}$ to all
orders and with integer coefficients, and in particular by taking the reciprocal
of \eqref{e:zNexpansion} we find that, for any $\lambda \in (0,\lambda_0]$,
\begin{equation}
    \mu_N
    = N - 1 - \frac{1}{N} - \frac{4}{N^2} - \frac{26}{N^3}
    + O\Big( \frac{1}{N^4}\Big) .
\end{equation}
The existence proof for the expansions for $z_N$ and $A_N$
presents an algorithm for the
computation of any number of coefficients, and
more terms could be computed with computer assistance as
has been done for $\Z^d$ (see
Section~\ref{sec:infinite-graphs}, in fact the hypercube computations appear to be
substantially easier than for $\Z^d$).

\subsubsection{Susceptibility and expected length}

The following theorem
provides upper and lower bounds on the susceptibility.
As the proof will show, the lower bound in \eqref{e:chiasy-new}
is a general consequence of submultiplicativity
and holds on any finite or infinite transitive graph,
while the upper bound relies on the proof of a ``bubble condition.''

\begin{theorem}
\label{thm:chiN-new}
Fix $\lambda \in (0,\lambda_0]$, assume that $\lambda V^{1/2} \ge 1$,
and let $z_N=z_N(\lambda)$.
Let $\beta=N^{-1}+\lambda^2$.
For all $z \in [0,z_N]$,
\begin{equation}
\label{e:chiasy-new}
    \frac{1}{\lambda^{-1} V^{-1/2} + 1-z/z_N}
    \le
    \chi_N(z)
    \le
    \frac{2-z/z_N}{\lambda^{-1} V^{-1/2} + (1-O(\beta))(1-z/z_N)}.
\end{equation}
\end{theorem}

The expected length of a self-avoiding walk is defined as follows.
The \emph{length} $L$ is the
discrete random variable with $z$-dependent probability mass function
\begin{equation}
\label{e:PzNdef}
    \P_z^{(N)}(L=n+1) = \frac{1}{\chi_N(z)} c_n^{(N)}z^n,
\end{equation}
with fixed $N$ and fixed $z \ge 0$, and for all nonnegative integers $n$.
With this definition using $n+1$ on the left-hand side
of \eqref{e:PzNdef}, $L$ reflects the number of vertices in the walk rather than
the number of steps.
The \emph{expected length} is
\begin{equation}
\label{e:ELdef}
    \E_z^{(N)} L
    = \sum_{n=0}^\infty (n+1) \P_z^{(N)}(L=n+1)
    = \frac{1}{\chi_N(z)} \partial_z [z\chi_N(z)].
\end{equation}
The next theorem
concerns the asymptotic behaviour of the expected length.
The upper bound is a consequence of submultiplicativity and holds on any
finite or infinite transitive graph, while the lower bound
is a consequence of the bubble condition.

\begin{theorem}
\label{thm:EL-new}
Fix $\lambda \in (0,\lambda_0]$, assume that $\lambda V^{1/2} \ge 1$, and let $z_N=z_N(\lambda)$.
Let $\beta=N^{-1}+\lambda^2$.
For $z \in [0, z_N]$,
\begin{equation}
\label{e:ELasy-new}
    [1-O(\beta)] \chi_N(z)
    \le
    \E_{z}^{(N)} L
    \le
    \chi_N(z)
    .
\end{equation}
In particular, at the critical value,
\begin{equation}
    \E_{z_N}^{(N)} L
    =
    \lambda V^{1/2} [1+O(\beta)].
\end{equation}
\end{theorem}

\subsection{Notation}

We write $f \sim g$ to mean $\lim f/g =1$, $f\prec g$ to mean $f \le c_1 g$ with $c_1>0$ and $f \succ  g$ to mean  $g \prec  f$. We also write $f \asymp g$ when $ g \prec f \prec g$.
Constants in these relations are not permitted to depend on $N$ but may depend on the
choice of $\lambda$ used to define $z_N$, and also on $p\in (0,\frac 12)$ when it
is part of the discussion.

\subsection{Conjectured phase transition}

In the hypotheses of Theorem~\ref{thm:dilute-short} it is assumed that $p \in (0, \frac 12)$. At the upper limit $p=\frac 12$, which Theorem~\ref{thm:dilute-short} does not address, the
error estimate is no longer small.  We believe that this is not an artifact of our proof but
that the asymptotic behaviour does change once $n$ reaches $V^{1/2}$.
The nature of this conjectured change can be anticipated by comparison with
self-avoiding walk on the complete graph, which is exactly solvable---its susceptibility
is essentially an incomplete Gamma function---and which has been analysed recently in
\cite{Slad20} (see also \cite{DGGNZ19}).  In \cite{Slad20}, it is conjectured that
the susceptibility $\chi_N(z)$ for the hypercube
remains of order $V^{1/2}$ throughout the \emph{critical
window} consisting of ($N$-dependent) $z$ values such that
$|1- z/z_N|$ is of order $V^{-1/2}$.   A related conjecture for self-avoiding
walk on a discrete torus of dimension $d>4$ is discussed in \cite{MS22}.

On the complete graph on $V$ vertices, the number $k_n$ of $n$-step self-avoiding walks starting
from a fixed vertex is simply
\begin{equation}
    k_n = \frac{v!}{(v-n)!},
\end{equation}
where $v=V-1$.
In the limit in which $v \to \infty$, and assuming for simplicity
that $n=o(v^{2/3})$ (so in particular $v-n\to \infty$), it follows from Stirling's formula
that
\begin{equation}
    k_n = v^n e^{-n^2/2v} [1+o(1)].
\end{equation}
We expect similar asymptotics to apply to the hypercube in and around the critical
window, with dominant behaviour
$\mu_N^n e^{-\alpha n^2/V}$ for $c_n^{(N)}$, for some $\alpha>0$.  This is consistent with
the susceptibility remaining of order $V^{1/2}$ in the critical window.

By analogy with the theory of self-avoiding walk on the complete graph
developed in detail in \cite{Slad20} (see also \cite{DGGNZ19}),  we are led
to the
conjecture for the hypercube that the interval $z \in (0,\infty)$ is
divided into three regimes.  With $z$ written as $z=z_N(1+\epsilon)$
with $\epsilon \in (-1,\infty)$, these regimes are:
\begin{itemize}
\item the \emph{dilute phase} $\epsilon \ll - V^{-1/2}$:
\[
    \chi_N \asymp \epsilon^{-1},
    \qquad
    c_n^{(N)} \sim A_N\mu_N^n \;\; \text{for $n \ll V^{1/2}$},
    \qquad
    \E_z^{(N)}L \asymp \epsilon^{-1} ;
\]
\item the \emph{critical window} $|\epsilon| \asymp V^{-1/2}$:
\[
     \chi_N \asymp V^{1/2}, \qquad
    c_n^{(N)} \asymp  \mu_N^n \;\; \text{for $n \asymp V^{1/2}$},
    \qquad
    \E_z^{(N)}L \asymp V^{1/2};
\]
\item the \emph{dense phase} $\epsilon \gg V^{-1/2}$:
\[
    \chi_N  \;\;\text{exponential in $V$},
    \qquad
    c_n^{(N)} \ll \mu_N^n \;\;\text{for $n \gg V^{1/2}$},
    \qquad
    \E_z^{(N)}L \asymp V \frac{\epsilon}{1+\epsilon}
    .
\]
\end{itemize}
In particular, if $\epsilon = V^{-p}$ with $p \in (0,\frac 12)$ then the above
states that $\E_z^{(N)}L \asymp V^{1-p}$, whereas if $\epsilon \ge c >0$
then it states that $\E_z^{(N)}L \asymp V $.
For the case $\epsilon = -V^{-p}$ with $p \in (0,\frac 12)$, the above states
that $\chi_N \asymp V^p \asymp \E_z^{(N)}L$.

Theorem~\ref{thm:chiN-new} proves the above behaviour for the susceptibility
in the dilute phase and in the critical window up to and including $z=z_N$.
Theorem~\ref{thm:dilute-short} proves the dilute behaviour of
$c_n^{(N)}$ as long as $n \le V^p$ for some $p< \frac 12$.
Theorem~\ref{thm:EL-new} proves the above behaviour for the expected length
in the dilute phase and in the critical window up to and including $z=z_N$.
It is an open problem to prove (or disprove) any of the remaining statements.

For general graphs, the mathematical analysis of the dense phase of self-avoiding walk is not yet very
well developed.  Various aspects of the dense phase are studied in \cite{BGJ05,DKY14,GI95,Yadi16}.

For percolation on the hypercube, a related and much-studied parallel to the above picture
is developed in  \cite{AKS82,BKL92,BCHSS04c,HH17book,HN20,HN17,HS05,HS06}.
Our analysis takes inspiration in particular from the general study of the percolation
phase transition on finite graphs including the hypercube from \cite{BCHSS05a}, though
we also rely on complex analytic methods that were not used for percolation.

\subsection{The connective constant on infinite graphs}
\label{sec:infinite-graphs}

It is something of a misnomer to refer to $\mu_N$ as the connective ``constant'' since
it depends on $N$ and also on the choice of $\lambda$.  However the terminology is natural
in the sense that on an infinite lattice the term ``connective constant'' is used for
the exponential growth rate for the number $c_n$ of $n$-step self-avoiding walks
started from a given vertex.
On any transitive graph, finite or infinite, $c_n$ obeys $c_{n+m}\le c_nc_m$
and by Fekete's lemma this implies existence of the limit
\begin{equation}
\label{e:cnlim}
    \mu = \lim_{n\to \infty} c_n^{1/n} = \inf_{n \ge 0}c_n^{1/n},
\end{equation}
where $\mu$ of course depends on the graph.
However on a finite graph, such as the hypercube,
$c_n$ is eventually zero so $\mu$ takes the
uninformative value $\mu=0$.
On an infinite lattice such as $\Z^d$ or the hexagonal lattice, $\mu$ is not zero
and it gives the exponential growth rate of $c_n$ in the sense of \eqref{e:cnlim}.
There are numerical estimates and rigorous bounds for the value of $\mu(\Z^d)$ but its
exact value is not known for any $d \ge 2$.
Exceptionally, for the hexagonal lattice it was predicted in \cite{Nien82} and proved
in \cite{D-CS12h} that $\mu({\rm Hex}) = \sqrt{2+\sqrt{2}}$.  Connective constants for more general graphs are
studied in \cite{AJ90,GL19,LL20,MW05,Pana19}.
Expansions for the connective constant have been considered
in other settings, e.g., two terms
were computed in
\cite{Pana19} for hyperbolic graphs.
The lace expansion (when applicable) provides a systematic method
for computation of many terms.

Indeed, for $\Z^d$ it is proved in \cite{HS95} that the connective constant has
an asymptotic expansion to all orders in $(2d)^{-1}$, with integer coefficients,
and in \cite{CLS07} thirteen of these coefficients are computed with the result that
\begin{eqnarray}
\label{e:mud1}
    \mu(\Z^d)  \hspace{-2mm} &=
    2d -1 -\frac{1}{2d} -\frac{3}{(2d)^2} -\frac{16}{(2d)^3} -
    \frac{102}{(2d)^4} -\frac{729}{(2d)^5}
    -\frac{5\,533}{(2d)^6}
    - \frac{42\,229}{(2d)^7}\nonumber\\
    & \qquad - \frac{288\,761}{(2d)^8}
- \frac{1\,026\,328}{(2d)^9}
+ \frac{21\,070\,667}{(2d)^{10}} + \frac{780\,280\,468}{(2d)^{11}}
+ O\big( \frac{1}{(2d)^{12}} \big).
\end{eqnarray}
Equivalently, the \emph{critical value} $z_c(\Z^d)=1/\mu(\Z^d)$ satisfies
\begin{eqnarray}
\label{e:zcd}
 z_c(\Z^d) \hspace{-2mm}  &= \frac{1}{2d} + \frac{1}{(2d)^2} + \frac{2}{(2d)^3}+ \frac{6}{(2d)^4}+
\frac{27}{(2d)^5}+ \frac{157}{(2d)^6}+ \frac{1\,065}{(2d)^7}
+ \frac{7\,865}{(2d)^8}+ \frac{59\,665}{(2d)^9}
\nonumber\\
&+ \frac{422\,421}{(2d)^{10}}+
\frac{1\,991\,163}{(2d)^{11}} -\frac{16\,122\,550}{(2d)^{12}} -\frac{805\,887\,918}{(2d)^{13}}
+ O\big( \frac{1}{(2d)^{14}} \big).
\end{eqnarray}
Also, in the asymptotic formula
$c_n = A\mu^n[1+O(n^{-\epsilon})]$ for $\Z^d$ with $d \ge 5$ proved in \cite{HS92a},
 the amplitude $A$
is proved in \cite{CLS07} to have an asymptotic
expansion to all orders, with integer coefficients, and in particular
\begin{eqnarray}
\label{A1d}
    A(\Z^d)  \hspace{-4mm}
    &=
     1   +\frac{1}{2d}  +\frac{4}{(2d)^2}  + \frac{23}{(2d)^3}
     +\frac{178}{(2d)^4}   +\frac{1\,591}{(2d)^5}  +
    \frac{15\,647}{(2d)^6}
    +   \frac{164\,766}{(2d)^7} + \frac{1\,825\,071}{(2d)^8}
     \nnb
     & \qquad
    +\frac{20\,875\,838}{(2d)^9}
    +\frac{240\,634\,600}{(2d)^{10}} +\frac{2\,684\,759\,873}{(2d)^{11}}
    +\frac{26\,450\,261\,391}{(2d)^{12}}
     + O\left( \frac{1}{(2d)^{13}}\right).
\end{eqnarray}

The possibility that the above series are Borel summable is investigated
but not resolved in  \cite{Grah10}.  See \cite{Soka80} for a sufficient
condition for Borel summability.  We believe that these series
and also the series for the hypercube in Theorem~\ref{thm:expansion}
have radius of convergence zero but are Borel summable;
to prove any of these statements is an open problem.
Numerical results of Pad\'e--Borel resummation \cite{KS01} of the above series
for $\mu(\Z^d)$ and $A(\Z^d)$ are reported in \cite[Table~15]{CLS07}.
For the related question
of the $1/d$ expansion for the critical point for the Berlin--Kac spherical model,
it is resolved affirmatively
in \cite{GF74} that the radius of convergence of the expansion
is zero.
There is a substantial literature concerning such $1/d$ expansions going
back as early as 1964 where the
first six coefficients of \eqref{e:mud1} were determined  \cite{FG64},
and decades later confirmed with rigorous error estimate \cite{HS95}.
Earlier expansions for the amplitude $A(\Z^d)$
including terms up to and including order $(2d)^{-2}$ (with rigorous error
estimate) and to $(2d)^{-5}$ (without rigorous error estimate) were given
respectively in \cite{HS95} and in \cite{Gaun86,NFID92}.

Such expansions have also been studied for other models including lattice animals \cite{MS13}
and percolation \cite{HS95,HS05,HS06}.  In particular, a theorem analogous to
Theorem~\ref{thm:expansion} is proved for the critical value of percolation on the
hypercube and on $\Z^d$, this time with
rational rather than integer coefficients, in \cite{HS05,HS06}.

\subsection{Organisation}

Sections~\ref{sec:chi}--\ref{sec:expansion} provide the proofs of Theorems~\ref{thm:cnbd}--\ref{thm:EL-new},
which are organised as follows.

In Section~\ref{sec:chiub} we state Proposition~\ref{prop:FzN} which gives a lower
bound on the reciprocal of the susceptibility as a function of complex $z$ in the disk $|z|\le z_N$,
where $z_N=z_N(\lambda)$ for a sufficiently small choice of $\lambda>0$.
In conjunction with the elementary Tauberian theorem stated in Lemma~\ref{lem:Tauberian},
this leads to a short proof of the general upper bound on $c_n^{(N)}$ stated in
Theorem~\ref{thm:cnbd}.
In Section~\ref{sec:dilute}, a version of Theorem~\ref{thm:dilute-short} with a more
accurate error estimate is stated as Theorem~\ref{thm:dilute}, and the proof of
Theorem~\ref{thm:dilute} is given subject to Propositions~\ref{prop:alphabeta}--\ref{prop:F}.
These two propositions give more refined information on the reciprocal of the susceptibility
than Proposition~\ref{prop:FzN}
but in a smaller disk $|z|\le z_N(1-V^{-p})$ for arbitrary but fixed $p \in (0,\frac 12)$.
This detailed information allows for the extraction of a leading term from the susceptibility,
and thereby from its coefficients $c_n^{(N)}$,
with an error that can be estimated using the Tauberian theorem.  This proves
Theorems~\ref{thm:cnbd}--\ref{thm:dilute-short} subject to the control of the
reciprocal of the susceptiblity stated in Propositions~\ref{prop:FzN},
\ref{prop:alphabeta}, and \ref{prop:F}, which are all proved using the lace expansion.

The lace expansion was introduced by Brydges and Spencer in 1985
to study weakly self-avoiding walk on $\Z^d$ in dimensions $d >4$  \cite{BS85}.
Since then, it has been developed into a flexible
method for the analysis of critical behaviour in
many high-dimensional settings, including
self-avoiding walk, lattice trees, lattice animals, percolation on finite and
infinite graphs, oriented percolation, the contact process,
and spin systems (Ising and $\varphi^4$ models).
In Section~\ref{sec:lace}, we review the lace expansion in
our present context of self-avoiding
walk on the hypercube.

The convergence of the lace expansion employs some elementary
estimates for simple random walk on the hypercube which are proved in Section~\ref{sec:rw}.
The convergence of the lace expansion is established in Section~\ref{sec:le-conv}
for complex $z$ in the disk $|z|\le z_N$, via the
Fourier approach used previously for percolation in \cite{BCHSS05b} and
adapted to self-avoiding walk in \cite{Slad06}.  The zero mode of the Fourier transform
plays a special and key role, and is what forces
the choice of a small $\lambda$ for the definition of
the critical value $z_N=z_N(\lambda)$.
The fact that we work on the hypercube results in a convergence proof that is strikingly simple.
The centrepiece for high-dimensional percolation is the triangle condition \cite{AN84,HH17book};
its role is played here by the bubble condition which is established in
Section~\ref{sec:bootstrap}.  The importance of the bubble condition for self-avoiding
walk goes back at least as far as \cite{BFF84}.  The bulk of our analysis would apply
generally to other transitive graphs for which the bubble condition holds.

Once the convergence of the lace expansion has been proved, it is short work in
Section~\ref{sec:manypfs} to prove Propositions~\ref{prop:FzN} and
\ref{prop:alphabeta}, as well as the estimates for the susceptibility and
expected length in Theorems~\ref{thm:chiN-new} and \ref{thm:EL-new}.
The proof of Proposition~\ref{prop:F} makes use of the fractional derivative
methodology developed in \cite{HS92a}, which is briefly reviewed in
Section~\ref{sec:fracder}, before proving Proposition~\ref{prop:F} in
Section~\ref{sec:pfpropF}.

Finally, in Section~\ref{sec:expansion} we prove the existence of the $1/N$ expansions
for $z_N$ and $A_N$
stated in Theorem~\ref{thm:expansion} and compute the first five coefficients.
The general approach to the existence proof is related to the approach used for $\Z^d$
in \cite{HS95}, but improvements to that approach which were introduced in \cite{CLS07}
are adapted here to the hypercube to obtain a relatively simple existence proof.  The computation
of the expansion coefficients follows a straightforward iterative procedure and could
be extended to more terms with further effort to enumerate lace graphs on the hypercube.
For small lace graphs,
enumeration on the hypercube is not difficult to adapt from the
enumerations on $\Z^d$ provided in \cite{CLS06arXiv}, and in this way
we avoid any difficult counting in the computation of the five coefficients
given in Theorem~\ref{thm:expansion}.

\section{Analysis of the susceptibility}
\label{sec:chi}

In this section, we prove Theorems~\ref{thm:cnbd} and \ref{thm:dilute-short} subject
to Proposition~\ref{prop:FzN} (for Theorem~\ref{thm:cnbd})
and Propositions~\ref{prop:alphabeta}--\ref{prop:F} (for Theorem~\ref{thm:dilute-short}).
These propositions give estimates on the susceptibility which can be converted into
estimates for $c_n^{(N)}$ via the Tauberian theorem in Lemma~\ref{lem:Tauberian}.

\subsection{Upper bound:  proof of Theorem~\ref{thm:cnbd}}
\label{sec:chiub}

\subsubsection{Use of the Tauberian theorem}

The susceptibility is a polynomial, so its reciprocal
\begin{equation}
    F_N(z) = \frac{1}{\chi_N(z)}
\end{equation}
is a meromorphic function of $z\in\C$.  Since $\chi_N$ is a polynomial with positive
coefficients, $F_N$ has no poles on the nonnegative real axis.
We will prove the following proposition in Section~\ref{sec:pfF} using the lace expansion.

\begin{prop}
\label{prop:FzN}
There is a $\lambda_0 >0$ such that, with $z_N=z_N(\lambda)$ for any
$\lambda \in (0,\lambda_0]$, and with $N$ sufficiently large depending on $\lambda$,
the function $F_N$ obeys
the bounds $|F_N'(z)|\le 2  N$ and $|F_N(z) | \ge \frac 12 |1-z/z_N|$
uniformly in $z\in\C$ with $|z|\le z_N$.  In addition, $z_N \le 2N^{-1}$.
\end{prop}

To prove Theorem~\ref{thm:cnbd},
we use Proposition~\ref{prop:FzN} in combination with the
Tauberian theorem from \cite[Theorem~4]{FO90} stated in the next lemma.

\begin{lemma}
\label{lem:Tauberian}
Let $b>1$.
Suppose that the power series
$f(z)=\sum_{n=0}^\infty a_n z^n$ obeys $|f(z)| \le K_1|1-z/\rho|^{-b}$
for all $|z|< \rho$.  Then   $|a_n| \le K_2K_1 n^{b-1}\rho^{-n}$
with $K_2$ depending only on $b$.
\end{lemma}

\begin{proof}[Proof of Theorem~\ref{thm:cnbd}]
By Proposition~\ref{prop:FzN},
\begin{equation}
    |\chi_N'(z)| = \left| \frac{F_N'(z)}{F_N(z)^2} \right| \le  \frac{8N}{|1-z/z_N|^2}
\end{equation}
holds uniformly in $|z| \le z_N$.
Since the coefficient of $z^n$ in $\chi_N'(z)$ is $(n+1)c_{n+1}^{(N)}$, it follows from
Lemma~\ref{lem:Tauberian} (with $f=\chi_N'$, $\rho=z_N$ and $b=2$)
that there is a constant $K$ such that
\begin{equation}
    (n+1)c_{n+1}^{(N)} \le K N n^{2-1} z_N^{-n}
\end{equation}
for all $n$.  Since $Nz_N \le 2$ by Proposition~\ref{prop:FzN},
\begin{equation}
\label{e:cnpf}
    c_n^{(N)} \le KN  z_N^{-(n-1)} = KNz_N   \mu_N^n \le2K \mu_N^n
\end{equation}
which is the desired upper bound.

In the above we have assumed that $N$ is sufficiently large, say $N \ge N_0(\lambda_0)$.
However, for $N<N_0$ there are only finitely many choices of $(n,N)$ and we can
therefore obtain \eqref{e:cnpf} for all $(n,N)$ (with $\lambda V^{1/2}\ge 1$)
by increasing $K$.
\end{proof}

\subsubsection{Remarks on Tauberian theorems}
\label{sec:Trk}

\begin{enumerate}
\item
Extensions of Lemma~\ref{lem:Tauberian} in \cite[Lemma~3.2]{DS98}
include the case $b=1$ which instead has upper bound $\rho^{-n}\log n$.
This is the reason why $\chi_N'$
appears rather than $\chi_N$ in the above application of
Lemma~\ref{lem:Tauberian} to obtain Theorem~\ref{thm:cnbd}:
applied directly to $\chi_N$, the extension to
Lemma~\ref{lem:Tauberian} would produce an unwanted logarithm in the upper bound.
Lemma~\ref{lem:Tauberian} is false for $b<1$, a counterexample is given in the Remark
following \cite[Lemma~6.3.3]{MS93}.

\item
We have chosen to prove Theorem~\ref{thm:cnbd} using Lemma~\ref{lem:Tauberian}
because Lemma~\ref{lem:Tauberian} is also required for the proof of Theorem~\ref{thm:dilute-short}.
However, for Theorem~\ref{thm:cnbd} we could instead have applied Hutchcroft's
Tauberian theorem \cite[Lemma~3.4]{Hutc19} for submultiplicative sequences (since
we do have $c_{n+m}^{(N)} \le c_{n}^{(N)}c_{m}^{(N)}$), which implies that
for all $n \ge 1$ and all $z \ge w >0$ it is the case that
\begin{align}
\label{e:TomTauberian}
    c_{n}^{(N)} &
    \le \frac{z^n}{w^{2n}} \left( \frac{\chi_N(w)}{n+1} \right)^2
    .
\end{align}
With the choices $z=z_N$ and $w=\frac{n}{n+1}z_N$, and with the upper bound
$\chi_N(w) \le 2 |1-w/z_N|^{-1}$ of Proposition~\ref{prop:FzN}, the upper bound
of  Theorem~\ref{thm:cnbd} follows from \eqref{e:TomTauberian} and without the need to consider
the derivative $\chi_N'$ nor to consider complex $z$.
However the application of Lemma~\ref{lem:Tauberian} cannot be replaced by
\cite[Lemma~3.4]{Hutc19} in
Section~\ref{sec:dilute} because the generating function used in that application is
not for a submultiplicative sequence, and also \eqref{e:TomTauberian} fails to provide sharp
powers of $n$
for generating functions that diverge faster than linearly.
\end{enumerate}

\subsection{Asymptotic formula: proof of Theorem~\ref{thm:dilute-short}}
\label{sec:dilute}

\subsubsection{Extended version of Theorem~\ref{thm:dilute-short}}

The following theorem, whose statement is not limited to $n \le V^p$
as in Theorem~\ref{thm:dilute-short}, implies
Theorem~\ref{thm:dilute-short}.

\begin{theorem}
\label{thm:dilute}
There exists $\lambda_0>0$ such that with $\mu_N=\mu_N(\lambda)$ defined by \eqref{e:mudef}
for any $\lambda \in (0,\lambda_0]$,
with any choice of $p \in (0,\frac 12)$ and $a\in (0,1)$, and for all $n,N \in \N$
(with $\lambda V^{1/2}\ge 1$),
\begin{equation}
\label{e:cnmr}
    c_{n}^{(N)} = A_N \mu_N^n
    \left[ 1   +O\big(  n^{-a} (N^{-1} + V^{(2+a)p-1} )\big)\right] \, \left[1 + O( V^{-p})  \right]^n
    .
\end{equation}
The sequence $A_N$ is independent of $n$ (but depends on $\lambda$ and $p$)
and obeys $A_N=1+O(N^{-1})$.
The constants in error terms depend on $p$, $a$ and $\lambda$.
\end{theorem}

Theorem~\ref{thm:dilute} has most significance for the largest values of $n$
which give a small error, so $n \le V^p$
for $p$ close to $\frac 12$.  To understand this,
consider first the factor $[1 + O( V^{-p})  ]^n$, which is bounded for $n \le V^p$ but
is not close to $1$ when $n=V^p$.  However when $n \le V^p$
we can also apply Theorem~\ref{thm:dilute} for any choice of $p'\in (p,\frac 12)$ and in this case
$V^{-p'}=(V^{-p})^{p'/p} \le n^{-p'/p}$ and hence
\begin{equation}
    [1 + O( V^{-p'})  ]^n
    =
    1+O(nV^{-p'}) \le 1 + O(n n^{-p'/p}) = 1 + O(n^{-(p'-p)/p}).
\end{equation}
Also, given any
$p'\in (\frac 13,\frac 12)$, we can choose $a=\frac{1-2p'}{p'}\in (0,1)$ in which
case $V^{(2+a)p'-1}=1$.
Thus, \eqref{e:cnmr} can be simplified in this case of $n \le V^p$ as
(with $p'$ and $a$ as above)
\begin{equation}
    c_{n}^{(N)} = A_N \mu_N^n
    \left[ 1   +O(  n^{-a}) + O(n^{-(p'-p)/p})\right]
    \qquad (n \le V^p)
    .
\end{equation}
Therefore, as long as $n \le V^p$ for some $p<\frac 12$, the leading asymptotic behaviour
of $c_{n}^{(N)}$ is $A_N \mu_N^n$ and hence $\mu_N$ is the
exponential growth rate in this regime.
In this way, Theorem~\ref{thm:dilute} implies Theorem~\ref{thm:dilute-short}.
We will therefore prove Theorem~\ref{thm:dilute}.
It suffices to consider $N$ large in the proof, since \eqref{e:cnmr} holds for any finite
set of $(n,N)$ by adjusting the constants.

\subsubsection{Proof of Theorem~\ref{thm:dilute}}

The proof of Theorem~\ref{thm:dilute}
also uses Lemma~\ref{lem:Tauberian}, but for this it is necessary to extract
leading behaviour and then apply the Tauberian theorem to bound the remainder term.
This requires an extension of Proposition~\ref{prop:FzN} in which the linear part
of $F_N$ is extracted with a higher-order remainder.
In this section, we reduce the proof of Theorem~\ref{thm:dilute} to
Propositions~\ref{prop:alphabeta}--\ref{prop:F}, which
are proved in Section~\ref{sec:pfF}
using the lace expansion.
We always assume that $N$ is large enough that $\lambda V^{1/2} \ge 1$ so that $z_N(\lambda)$
is well defined.

To extract the linear term, our method gives useful results only if we restrict
$z$ to a smaller disk than the disk $|z|\le z_N$ of Proposition~\ref{prop:FzN}.
Thus, for $p>0$, we define $\zeta_p=\zeta_p(N,\lambda) >0$  by
\begin{equation}
\label{e:zetapdef}
    \zeta_p = z_N(\lambda) (1-V^{-p}),
\end{equation}
and we will work in the disk $|z|\le \zeta_p$.  It will be necessary
to restrict to $p\in (0,\frac 12)$.
The linear approximation to $F_N(z)$ near $\zeta_p$ is the linear function
\begin{equation}
    \Phi_N(z) = F_N(\zeta_p) + F_N'(\zeta_p)(z-\zeta_p),
\end{equation}
with remainder
\begin{equation}
    R_N(z) = F_N(z) - \Phi_N(z).
\end{equation}
Thus we have
\begin{equation}
    \chi_N(z) = \frac{1}{F_N(z)} = \frac{1}{\Phi_N(z)} +H_N(z),
    \qquad H_N(z) = - \frac{R_N(z)}{\Phi_N(z)F_N(z)} .
\end{equation}
We write the coefficients of the power series representations of $1/\Phi_N(z)$ and
$H(z)$ as
\begin{equation}
    \frac{1}{\Phi_N(z)} = \sum_{n=0}^\infty \varphi_n z^n,
    \qquad
    H_N(z)  = \sum_{n=0}^\infty h_n z^n .
\end{equation}
Both $\varphi_n$ and $h_n$ depend on $N$.
By definition,
\begin{equation}
    c_n^{(N)} = \varphi_n + h_n.
\end{equation}
The next proposition provides what is needed for good estimates on the linear approximation
$\Phi_N$ to $F_N$.

\begin{prop}
\label{prop:alphabeta}
There is a $\lambda_0>0$ such that for any $\lambda\in (0, \lambda_0]$,
 for any $p \in (0,\frac 12)$,
with
$\zeta_p=z_N(\lambda)(1-V^{-p})$, and with $\lambda$-dependent error bounds,
\begin{equation}
    \zeta_p = N^{-1}[1+O(N^{-1})], \qquad
    F_N(\zeta_p) \asymp V^{-p} , \qquad F_N'(\zeta_p) = -N +O(1).
\end{equation}
\end{prop}

The next proposition gives a bound on the remainder term $R_N(z)$
and its derivative in the disk $|z|\le \zeta_p$ in the complex plane.

\begin{prop}
\label{prop:F}
There is a $\lambda_0>0$ such that for any $\lambda\in (0, \lambda_0]$,
any $p \in (0,\frac 12)$, any $a\in (0,1)$,
 any  $z\in\C$ with $|z| \le \zeta_p=z_N(\lambda)(1-V^{-p})$, and with $\lambda$-dependent error bounds,
\begin{align}
\label{e:frac-der-OK}
    |R_N(z)| &\prec  N^{-1} (1 + NV^{(2+a)p-1}) |1-z/\zeta_p|^{1+a}   ,
    \\
\label{e:varphiderbd}
    |R_N'(z)|  &\prec  (1 +NV^{(2+a)p-1}) |1-z/\zeta_p|^a
    .
\end{align}
\end{prop}

 An indication of deterioration for
$p \ge \frac 12 $ can be seen from the term $V^{(2+a)p-1}$ in \eqref{e:frac-der-OK}.
We desire a remainder $R_N$ of higher order than linear, so $a>0$, and when $p \ge \frac 12$
the term $V^{(2+a)p-1}$ grows exponentially in $N$ and spoils control unless $a \le 0$
which we do not permit.

We prove Theorem~\ref{thm:dilute} by using the Tauberian theorem
Lemma~\ref{lem:Tauberian} in conjunction with
Propositions~\ref{prop:alphabeta}--\ref{prop:F}.
To prepare for this we have the following two
corollaries of the above propositions.  The first corollary is for the leading
behaviour of $c_n^{(N)}$.

\begin{cor}
\label{cor:varphin}
With $z_N=z_N(\lambda)$ for $\lambda\in (0,\lambda_0]$, for  all
$p \in (0,\frac 12)$, all $N$ sufficiently large, and
all $n\in \N$,
the coefficient $\varphi_n$ of $z^n$ in $1/\Phi_N(z)$ obeys
\begin{equation}
    \varphi_n =
    A_N \mu_N^n (1+O(V^{-p}))^n
\end{equation}
with $A_N= 1+O(N^{-1})$ independent of $n$ (but dependent on $p$ and $\lambda$),
and with $\lambda$-dependent error bounds.
\end{cor}

\begin{proof}
We define
\begin{equation}
    \alpha_N = F_N(\zeta_p) - \zeta_p F_N'(\zeta_p), \qquad \beta_N = - F_N'(\zeta_p),
\end{equation}
which are both positive for large $N$ since then $F_N'(\zeta_p)$ is negative
by Proposition~\ref{prop:alphabeta}.
By definition,
$\Phi(z) = \alpha_N - \beta_N z$, so
expansion of the geometric series gives
\begin{equation}
    \frac{1}{\Phi_N(z)}
    =
    \frac{1}{\alpha_N }\sum_{n=0}^\infty \left(\frac{\beta_N z}{\alpha_N}\right)^n
    \qquad
    (|z| < \alpha_N/\beta_N)
    .
\end{equation}
Let $A_N=1/\alpha_N$, so $A_N$ depends on $\lambda$ and $p$.  Then $A_N=1+O(N^{-1})$ by Proposition~\ref{prop:alphabeta}
and we have
\begin{equation}
    \varphi_n = A_N \left(\frac{\beta_N}{\alpha_N}\right)^n.
\end{equation}
By definition and by Proposition~\ref{prop:alphabeta},
\begin{align}
    \frac{\beta_N}{\alpha_N}
    &= \mu_N \frac{z_N\beta_N}{\alpha_N}
    =
    \mu_N \frac{1}{1-V^{-p}}\frac{\zeta_p\beta_N}{\alpha_N}
    \nnb &=
    \mu_N \frac{1}{1-V^{-p}}\frac{1}{1 + \frac{F_N(\zeta_p)}{- \zeta_p F_N'(\zeta_p)}}
    =\mu_N (1+O(V^{-p})).
\label{e:abratio}
\end{align}
This gives the desired result
\begin{equation}
    \varphi_n =
    A_N \left(\frac{\beta_N }{\alpha_N}\right)^n
    =
    A_N \mu_N^n (1+O(V^{-p}))^n
\end{equation}
and the proof is complete.
\end{proof}

To prove Theorem~\ref{thm:dilute}, it now suffices to prove that
\begin{equation}
\label{e:hmu}
     h_n \mu_N^{-n} =
    O( n^{-a}(N^{-1}+V^{(2+a)p-1})) (1+O(V^{-p}))^n.
\end{equation}
To do so, we will use the following corollary of Propositions~\ref{prop:FzN} and
\ref{prop:alphabeta}--\ref{prop:F}.

\begin{cor}
\label{cor:susceptibility}
With $z_N=z_N(\lambda)$ for $\lambda\in (0,\lambda_0]$, and for all
$p \in (0,\frac 12)$, all $a\in (0,1)$, and
all $z\in \C$ with $|z| \le \zeta_p$,
\begin{equation}
\label{e:Hprime}
    |H_N'(z)| \prec \frac{1+NV^{(2+a)p-1} }{|1-z/\zeta_p|^{2-a}}.
\end{equation}
\end{cor}

\begin{proof}
Let $|z| \le \zeta_p$.
By definition,
\begin{equation}
\label{e:Hder}
    H_N'(z) = -\frac{R_N'(z)} {F_N(z) \Phi_N(z)}
    + \frac{R_N(z)F_N'(z)} {F_N(z)^2 \Phi_N(z)} -
    \frac{\beta_N R_N(z)} {F_N(z) \Phi_N(z)^2}.
\end{equation}
To bound the denominators of \eqref{e:Hder} we proceed as follows.
With the notation from the proof of Corollary~\ref{cor:varphin},
it follows from the facts that $\beta_N \zeta_p \sim 1$ and
\begin{equation}
    \frac{\alpha_N}{\zeta_p\beta_N} = 1 + \frac{F_N(\zeta_p)}{-\zeta_pF_N'(\zeta_p)} \ge 1
\end{equation}
that for $|z| \le \zeta_p$ and for large $N$ we have
\begin{align}
\label{e:disk}
    |\Phi_N(z)| & = |\alpha_N-\beta_Nz|
    = \zeta_p\beta_N \Big|\frac{\alpha_N}{\zeta_p\beta_N} - \frac{z}{\zeta_p} \Big|
    \ge \frac 12 |1 - z/\zeta_p |,
\end{align}
where in the last inequality we used the geometric fact that if $a>1$
and $|w| \le 1$ then
$|a-w|\ge |1-w|$.
Similarly, it follows from the linear lower bound on $F_N$ from
Proposition~\ref{prop:FzN} that on the disk $|z| \le \zeta_p$ we have
\begin{align}
\label{e:varphilb}
    |F_N(z)| \succ z_N^{-1}|z_N-z| \ge z_N^{-1}|\zeta_p -z|
    \succ |1-z/\zeta_p|
\end{align}
where we used $z_N > \zeta_p$ for the second inequality and $z_N \sim \zeta_p$ for the third.
Proposition~\ref{prop:FzN} also gives $|F_N'(z)| \prec N$.
Therefore, by \eqref{e:Hder}, \eqref{e:disk}, \eqref{e:varphilb},
$\beta_N \le 2 N$, and Proposition~\ref{prop:F},
\begin{align}
\label{e:Hderbd}
    |H_N'(z)|
    & \prec
    \frac{|R_N'(z)|} {|1-z/\zeta_p|^2}
    + \frac{N|R_N(z)|} {|1-z/\zeta_p|^3}
    \nnb &\prec
    \frac{1 +NV^{(2+a)p-1}}{|1-z/\zeta_p|^{2-a}}
    +
    \frac{1+NV^{(2+a)p-1}}{|1-z/\zeta_p|^{2-a}}
    ,
\end{align}
and the proof is complete.
\end{proof}

We now apply Lemma~\ref{lem:Tauberian} to prove Theorem~\ref{thm:dilute}.
Corollary~\ref{cor:susceptibility} is formulated for $H'$, rather than for $H$, for the
reason mentioned in the first remark of Section~\ref{sec:Trk}.

\begin{proof}[Proof of Theorem~\ref{thm:dilute}]
The combination of Corollary~\ref{cor:susceptibility} with Lemma~\ref{lem:Tauberian}
(with $f=H'$, $\rho=\zeta_p$, $b=2-a$)
immediately gives
\begin{equation}
    (n+1)|h_{n+1}| \prec
    n^{2-a-1}
    \zeta_p^{-n} (1 +NV^{(2+a)p-1}).
\end{equation}
With $\zeta_p \sim N^{-1}$, this implies that
\begin{equation}
    |h_{n}| \prec
    n^{-a}
    \zeta_p^{-(n-1)} (1 +NV^{(2+a)p-1})
    \prec
     n^{-a}
    \zeta_p^{-n} (N^{-1} +V^{(2+a)p-1}).
\end{equation}
It suffices now to observe that
\begin{equation}
    (\mu_N \zeta_p)^{-n} =   (1-  V^{-p} )^{-n}.
\end{equation}
This proves \eqref{e:hmu} and therefore completes the proof.
\end{proof}

Thus to prove Theorem~\ref{thm:dilute} (and thereby prove
Theorem~\ref{thm:dilute-short}) it suffices to prove Propositions~\ref{prop:alphabeta}--\ref{prop:F}.
The proofs of  Propositions~\ref{prop:FzN}, \ref{prop:alphabeta} and \ref{prop:F},
as well as of Theorems~\ref{thm:expansion}--\ref{thm:EL-new},
are based on the lace expansion
for self-avoiding walk, which we discuss next.

\section{The lace expansion}
\label{sec:lace}

In this section,
we summarise the derivation of the lace expansion as well as its diagrammatic estimates.
More extensive treatments can be found in
the original paper by Brydges and Spencer \cite{BS85} or in the books \cite{MS93,Slad06}.
The setting in those references is $\Z^d$ rather than the hypercube but the differences
for the derivation of the expansion and for its diagrammatic estimates in these two settings
are merely superficial.
Although we do not adopt this perspective here, the lace expansion can alternatively be
understood as arising from repeated application of the inclusion-exclusion relation
(see \cite[Section~5.1]{MS93}).

\subsection{Fourier transform on the hypercube}

The proofs of Propositions~\ref{prop:FzN},
\ref{prop:alphabeta} and \ref{prop:F} rely
heavily on Fourier transformation on the hypercube.
Given a function $f:\Q^N \to \C$, its Fourier transform is
\begin{equation}
\label{e:FTdef}
    \hat f(k) =  \sum_{x \in \Q^N} f(x) (-1)^{k\cdot x}
    \qquad
    (k \in \Q^N),
\end{equation}
where the dot product is defined by $k\cdot x = \sum_{i=1}^N k_ix_i$
with $k_i$ and $x_i$ respectively the $i^{\rm th}$ components
of $k$ and $x$.  The inverse
Fourier transform is
\begin{equation}
\label{e:IFT}
    f(x) = \frac{1}{V} \sum_{k \in \Q^N} \hat f(k) (-1)^{k\cdot x}\qquad
    (x \in \Q^N).
\end{equation}
The convolution $(f*g)(x) = \sum_{y\in\Q^N}f(x-y)g(y)$ obeys  $\widehat{f*g}=\hat f \hat g$.

An important example is when $f$ is the transition probability $D$ for simple random walk,
defined by
\begin{equation}
\label{e:Ddef}
    D(x) =
    \begin{cases}
    N^{-1} & |x|=1
    \\
    0 & |x| \neq 1
    \end{cases}
    \qquad (x \in \Q^N).
\end{equation}
Its Fourier transform is
\begin{equation}
\label{e:Dhat}
    \hat D(k) = 1 - \frac{2|k|}{N}  \qquad (k \in \Q^N).
\end{equation}

\subsection{The recursion relation}
\label{sec:rr}

Let $c_0(x) = \delta_{0,x}$, and, for $n \geq 1$,
let $c_n(x)$ denote the number of $n$-step self-avoiding walks that begin at the
origin and end at $x \in \Q^N$.
The \emph{two-point function} is the generating function for the sequence $c_n^{(N)}(x)$,
defined by
\begin{equation}
    G_z(x) = \sum_{n=0}^\infty c_n^{(N)}(x) z^n \qquad (x \in \Q^N,\; z\in \C).
\end{equation}
Since $c_n(x) =0$ for all $n \ge V$, the two-point function is a polynomial in $z$.

For $m \ge 2$, the lace expansion produces a function $\pi_m: \Q^N \to \Z$,
which we will define below.  We write its generating function, which is not a polynomial,
as $\Pi_z(x) = \sum_{m=2}^\infty \pi_m(x) z^m$.
The following proposition is a statement of the lace expansion.
The detailed derivation of the formulas in Proposition~\ref{prop:lace} can be
found in \cite[(3.14), (3.27)]{Slad06}, following the original proof in \cite{BS85}.
Although these references are for $\Z^d$,
the discussion in \cite{Slad06} applies verbatim
to the hypercube, or indeed to any finite or infinite
transitive graph after suitable adaptation of the
transition function $D$.
Each of $D, G_z,\pi_m, \Pi_z$ depends on
$N$ but to lighten the notation we do not make this explicit.

\begin{prop}
\label{prop:lace}
For $n \ge 1$ and for $x \in \Q^N$,
\begin{equation}
\label{e:cnrec1}
    c_n^{(N)}(x) = N(D * c_{n-1}^{(N)})(x) + \sum_{m=2}^n (\pi_m * c_{n-m}^{(N)})(x),
\end{equation}
and hence, for $z \in \C$ such that $\Pi_z(x)$ converges for all $x$,
\begin{equation}
\label{e:Gzx}
    G_z(x) = \delta_{0,x} + zN(D *G_z)(x) + (\Pi_z * G_z)(x).
\end{equation}
\end{prop}

Consequently,
\begin{equation}
    \hat G_z(k) = 1 + zN\hat D (k)\hat G_z(k) + \hat \Pi_z(k) \hat G_z(k).
\end{equation}
This can be rewritten as
\begin{equation}
\label{e:GFk}
    \hat G_z(k) = \frac{1}{1-zN\hat D(k) - \hat\Pi_z(k)}.
\end{equation}
Since the susceptibility is equal to $\chi_N(z)=\hat G_z(0)$, we obtain the identity
\begin{equation}
\label{e:chiPi}
    \chi_N(z) = \frac{1}{F_N(z)} = \frac{1}{1-zN  - \hat\Pi_z(0)}
\end{equation}
which is central to the proof of our main results
Theorems~\ref{thm:cnbd}--\ref{thm:EL-new}.
In order to make use of \eqref{e:GFk} and
\eqref{e:chiPi}, it will be necessary to obtain good estimates
on $\hat\Pi_z(k)$.  These will be achieved via diagrammatic estimates
in Section~\ref{sec:diag-est}, where
the convergence of $\Pi_z(x)$ will be studied.

\subsection{Graphs and laces}

The derivation of \eqref{e:cnrec1} uses the following definitions.
More detailed discussion and interpretation of these definitions can be
found in \cite[Section~3.3]{Slad06}.

\begin{defn}
\label{def-graph}
{\rm (i)}
Given an interval $I = [a,b]$ of positive integers, an \emph{edge} is a pair
$\{ s, t\}$
of elements of $I$, often written $st$ (with $s<t$).
A set of edges (possibly the empty set) is called a \emph{graph}.
Let $\Bcal[a,b]$ denote the set of all graphs.
\newline
{\rm (ii)}
A graph $\Gamma$ is \emph{connected}
if both $a$ and $b$ are
endpoints of edges in $\Gamma$, and if in addition, for any $c \in (a,b)$,
there is an edge $st \in \Gamma$ such that $s < c < t$.
Let $\Gcal[a,b]$ denote the set of all connected graphs on $[a,b]$.
\newline
{\rm (iii)}
A \emph{lace}
is a minimally connected graph: a connected graph for which
the removal of any edge would result in a disconnected graph.  The set of
laces on $[a,b]$
is denoted by $\Lcal [a,b]$, and the set of laces on
$[a,b]$ which consist of exactly $M$ edges is denoted $\Lcal ^{(M)} [a,b]$.
Figure~\ref{fig:LM} shows laces in $\Lcal^{(M)}[0,m]$ for $M=1,2,3,4$.
\end{defn}

The above definition of connectivity  is not the usual notion
of path-connectivity in graph
theory.  Instead, connected graphs are those $\Gamma$
for which the union of open real intervals
$\cup_{st \in \Gamma}(s,t)$ is equal to the connected
interval $(a,b)$.  This is the useful concept of connectivity for the lace expansion.

\begin{figure}
\begin{center}
\setlength{\unitlength}{0.0090in}%
\begin{picture}(215,150)(20,580)
\thicklines
\put(-100,680){\makebox(0,0)[lb]{\raisebox{0pt}[0pt][0pt]{\circle*{2}}}}
\put(-90,680){\makebox(0,0)[lb]{\raisebox{0pt}[0pt][0pt]{\circle*{2}}}}
\put(-80,680){\makebox(0,0)[lb]{\raisebox{0pt}[0pt][0pt]{\circle*{2}}}}
\put(-70,680){\makebox(0,0)[lb]{\raisebox{0pt}[0pt][0pt]{\circle*{2}}}}
\put(-60,680){\makebox(0,0)[lb]{\raisebox{0pt}[0pt][0pt]{\circle*{2}}}}
\put(-50,680){\makebox(0,0)[lb]{\raisebox{0pt}[0pt][0pt]{\circle*{2}}}}
\put(-40,680){\makebox(0,0)[lb]{\raisebox{0pt}[0pt][0pt]{\circle*{2}}}}
\put(-30,680){\makebox(0,0)[lb]{\raisebox{0pt}[0pt][0pt]{\circle*{2}}}}
\put(-20,680){\makebox(0,0)[lb]{\raisebox{0pt}[0pt][0pt]{\circle*{2}}}}
\put(-10,680){\makebox(0,0)[lb]{\raisebox{0pt}[0pt][0pt]{\circle*{2}}}}
\put(0,680){\makebox(0,0)[lb]{\raisebox{0pt}[0pt][0pt]{\circle*{2}}}}
\put(10,680){\makebox(0,0)[lb]{\raisebox{0pt}[0pt][0pt]{\circle*{2}}}}
\put(20,680){\makebox(0,0)[lb]{\raisebox{0pt}[0pt][0pt]{\circle*{2}}}}
\put(30,680){\makebox(0,0)[lb]{\raisebox{0pt}[0pt][0pt]{\circle*{2}}}}
\put(40,680){\makebox(0,0)[lb]{\raisebox{0pt}[0pt][0pt]{\circle*{2}}}}
\put(50,680){\makebox(0,0)[lb]{\raisebox{0pt}[0pt][0pt]{\circle*{2}}}}
\put(60,680){\makebox(0,0)[lb]{\raisebox{0pt}[0pt][0pt]{\circle*{2}}}}
\put(70,680){\makebox(0,0)[lb]{\raisebox{0pt}[0pt][0pt]{\circle*{2}}}}
\put(80,680){\makebox(0,0)[lb]{\raisebox{0pt}[0pt][0pt]{\circle*{2}}}}
\put(90,680){\makebox(0,0)[lb]{\raisebox{0pt}[0pt][0pt]{\circle*{2}}}}
\qbezier(-100,680)(-5,710)(90,680)
\put(-100,660){\makebox(0,0)[lb]{\raisebox{0pt}[0pt][0pt]{$s_1$}}}
\put(90,660){\makebox(0,0)[lb]{\raisebox{0pt}[0pt][0pt]{$t_1$}}}

\put(140,680){\makebox(0,0)[lb]{\raisebox{0pt}[0pt][0pt]{\circle*{2}}}}
\put(150,680){\makebox(0,0)[lb]{\raisebox{0pt}[0pt][0pt]{\circle*{2}}}}
\put(160,680){\makebox(0,0)[lb]{\raisebox{0pt}[0pt][0pt]{\circle*{2}}}}
\put(170,680){\makebox(0,0)[lb]{\raisebox{0pt}[0pt][0pt]{\circle*{2}}}}
\put(180,680){\makebox(0,0)[lb]{\raisebox{0pt}[0pt][0pt]{\circle*{2}}}}
\put(190,680){\makebox(0,0)[lb]{\raisebox{0pt}[0pt][0pt]{\circle*{2}}}}
\put(200,680){\makebox(0,0)[lb]{\raisebox{0pt}[0pt][0pt]{\circle*{2}}}}
\put(210,680){\makebox(0,0)[lb]{\raisebox{0pt}[0pt][0pt]{\circle*{2}}}}
\put(220,680){\makebox(0,0)[lb]{\raisebox{0pt}[0pt][0pt]{\circle*{2}}}}
\put(230,680){\makebox(0,0)[lb]{\raisebox{0pt}[0pt][0pt]{\circle*{2}}}}
\put(240,680){\makebox(0,0)[lb]{\raisebox{0pt}[0pt][0pt]{\circle*{2}}}}
\put(250,680){\makebox(0,0)[lb]{\raisebox{0pt}[0pt][0pt]{\circle*{2}}}}
\put(260,680){\makebox(0,0)[lb]{\raisebox{0pt}[0pt][0pt]{\circle*{2}}}}
\put(270,680){\makebox(0,0)[lb]{\raisebox{0pt}[0pt][0pt]{\circle*{2}}}}
\put(280,680){\makebox(0,0)[lb]{\raisebox{0pt}[0pt][0pt]{\circle*{2}}}}
\put(290,680){\makebox(0,0)[lb]{\raisebox{0pt}[0pt][0pt]{\circle*{2}}}}
\put(300,680){\makebox(0,0)[lb]{\raisebox{0pt}[0pt][0pt]{\circle*{2}}}}
\put(310,680){\makebox(0,0)[lb]{\raisebox{0pt}[0pt][0pt]{\circle*{2}}}}
\put(320,680){\makebox(0,0)[lb]{\raisebox{0pt}[0pt][0pt]{\circle*{2}}}}
\put(330,680){\makebox(0,0)[lb]{\raisebox{0pt}[0pt][0pt]{\circle*{2}}}}
\qbezier(140,680)(200,710)(260,680)
\qbezier(200,680)(265,710)(330,680)
\put(140,660){\makebox(0,0)[lb]{\raisebox{0pt}[0pt][0pt]{$s_1$}}}
\put(200,660){\makebox(0,0)[lb]{\raisebox{0pt}[0pt][0pt]{$s_2$}}}
\put(260,660){\makebox(0,0)[lb]{\raisebox{0pt}[0pt][0pt]{$t_1$}}}
\put(330,660){\makebox(0,0)[lb]{\raisebox{0pt}[0pt][0pt]{$t_2$}}}

\put(-100,600){\makebox(0,0)[lb]{\raisebox{0pt}[0pt][0pt]{\circle*{2}}}}
\put(-90,600){\makebox(0,0)[lb]{\raisebox{0pt}[0pt][0pt]{\circle*{2}}}}
\put(-80,600){\makebox(0,0)[lb]{\raisebox{0pt}[0pt][0pt]{\circle*{2}}}}
\put(-70,600){\makebox(0,0)[lb]{\raisebox{0pt}[0pt][0pt]{\circle*{2}}}}
\put(-60,600){\makebox(0,0)[lb]{\raisebox{0pt}[0pt][0pt]{\circle*{2}}}}
\put(-50,600){\makebox(0,0)[lb]{\raisebox{0pt}[0pt][0pt]{\circle*{2}}}}
\put(-40,600){\makebox(0,0)[lb]{\raisebox{0pt}[0pt][0pt]{\circle*{2}}}}
\put(-30,600){\makebox(0,0)[lb]{\raisebox{0pt}[0pt][0pt]{\circle*{2}}}}
\put(-20,600){\makebox(0,0)[lb]{\raisebox{0pt}[0pt][0pt]{\circle*{2}}}}
\put(-10,600){\makebox(0,0)[lb]{\raisebox{0pt}[0pt][0pt]{\circle*{2}}}}
\put(0,600){\makebox(0,0)[lb]{\raisebox{0pt}[0pt][0pt]{\circle*{2}}}}
\put(10,600){\makebox(0,0)[lb]{\raisebox{0pt}[0pt][0pt]{\circle*{2}}}}
\put(20,600){\makebox(0,0)[lb]{\raisebox{0pt}[0pt][0pt]{\circle*{2}}}}
\put(30,600){\makebox(0,0)[lb]{\raisebox{0pt}[0pt][0pt]{\circle*{2}}}}
\put(40,600){\makebox(0,0)[lb]{\raisebox{0pt}[0pt][0pt]{\circle*{2}}}}
\put(50,600){\makebox(0,0)[lb]{\raisebox{0pt}[0pt][0pt]{\circle*{2}}}}
\put(60,600){\makebox(0,0)[lb]{\raisebox{0pt}[0pt][0pt]{\circle*{2}}}}
\put(70,600){\makebox(0,0)[lb]{\raisebox{0pt}[0pt][0pt]{\circle*{2}}}}
\put(80,600){\makebox(0,0)[lb]{\raisebox{0pt}[0pt][0pt]{\circle*{2}}}}
\put(90,600){\makebox(0,0)[lb]{\raisebox{0pt}[0pt][0pt]{\circle*{2}}}}
\put(-100,580){\makebox(0,0)[lb]{\raisebox{0pt}[0pt][0pt]{$s_1$}}}
\put(-70,580){\makebox(0,0)[lb]{\raisebox{0pt}[0pt][0pt]{$s_2$}}}
\put(-40,580){\makebox(0,0)[lb]{\raisebox{0pt}[0pt][0pt]{$t_1$}}}
\put(0,580){\makebox(0,0)[lb]{\raisebox{0pt}[0pt][0pt]{$s_3$}}}
\put(40,580){\makebox(0,0)[lb]{\raisebox{0pt}[0pt][0pt]{$t_2$}}}
\put(90,580){\makebox(0,0)[lb]{\raisebox{0pt}[0pt][0pt]{$t_3$}}}
\qbezier(-100,600)(-70,630)(-40,600)
\qbezier(-70,600)(-15,630)(40,600)
\qbezier(0,600)(45,630)(90,600)

\put(140,580){\makebox(0,0)[lb]{\raisebox{0pt}[0pt][0pt]{$s_1$}}}
\put(170,580){\makebox(0,0)[lb]{\raisebox{0pt}[0pt][0pt]{$s_2$}}}
\put(200,580){\makebox(0,0)[lb]{\raisebox{0pt}[0pt][0pt]{$t_1$}}}
\put(230,580){\makebox(0,0)[lb]{\raisebox{0pt}[0pt][0pt]{$s_3$}}}
\put(260,580){\makebox(0,0)[lb]{\raisebox{0pt}[0pt][0pt]{$t_2$}}}
\put(290,580){\makebox(0,0)[lb]{\raisebox{0pt}[0pt][0pt]{$s_4$}}}
\put(310,580){\makebox(0,0)[lb]{\raisebox{0pt}[0pt][0pt]{$t_3$}}}
\put(330,580){\makebox(0,0)[lb]{\raisebox{0pt}[0pt][0pt]{$t_4$}}}
\qbezier(140,600)(170,630)(200,600)
\qbezier(170,600)(215,630)(260,600)
\qbezier(230,600)(270,630)(310,600)
\qbezier(290,600)(310,630)(330,600)
\put(140,600){\makebox(0,0)[lb]{\raisebox{0pt}[0pt][0pt]{\circle*{2}}}}
\put(150,600){\makebox(0,0)[lb]{\raisebox{0pt}[0pt][0pt]{\circle*{2}}}}
\put(160,600){\makebox(0,0)[lb]{\raisebox{0pt}[0pt][0pt]{\circle*{2}}}}
\put(170,600){\makebox(0,0)[lb]{\raisebox{0pt}[0pt][0pt]{\circle*{2}}}}
\put(180,600){\makebox(0,0)[lb]{\raisebox{0pt}[0pt][0pt]{\circle*{2}}}}
\put(190,600){\makebox(0,0)[lb]{\raisebox{0pt}[0pt][0pt]{\circle*{2}}}}
\put(200,600){\makebox(0,0)[lb]{\raisebox{0pt}[0pt][0pt]{\circle*{2}}}}
\put(210,600){\makebox(0,0)[lb]{\raisebox{0pt}[0pt][0pt]{\circle*{2}}}}
\put(220,600){\makebox(0,0)[lb]{\raisebox{0pt}[0pt][0pt]{\circle*{2}}}}
\put(230,600){\makebox(0,0)[lb]{\raisebox{0pt}[0pt][0pt]{\circle*{2}}}}
\put(240,600){\makebox(0,0)[lb]{\raisebox{0pt}[0pt][0pt]{\circle*{2}}}}
\put(250,600){\makebox(0,0)[lb]{\raisebox{0pt}[0pt][0pt]{\circle*{2}}}}
\put(260,600){\makebox(0,0)[lb]{\raisebox{0pt}[0pt][0pt]{\circle*{2}}}}
\put(270,600){\makebox(0,0)[lb]{\raisebox{0pt}[0pt][0pt]{\circle*{2}}}}
\put(280,600){\makebox(0,0)[lb]{\raisebox{0pt}[0pt][0pt]{\circle*{2}}}}
\put(290,600){\makebox(0,0)[lb]{\raisebox{0pt}[0pt][0pt]{\circle*{2}}}}
\put(300,600){\makebox(0,0)[lb]{\raisebox{0pt}[0pt][0pt]{\circle*{2}}}}
\put(310,600){\makebox(0,0)[lb]{\raisebox{0pt}[0pt][0pt]{\circle*{2}}}}
\put(320,600){\makebox(0,0)[lb]{\raisebox{0pt}[0pt][0pt]{\circle*{2}}}}
\put(330,600){\makebox(0,0)[lb]{\raisebox{0pt}[0pt][0pt]{\circle*{2}}}}
\end{picture}
\end{center}
\caption{\label{fig:LM} Laces in $\Lcal^{(M)}[0,m]$
for $M = 1,2,3,4$, with $s_1=0$ and $t_M=m$.}
\end{figure}
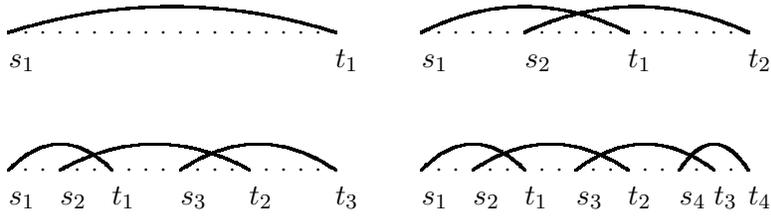

A lace
$L \in \Lcal^{(M)}[a,b]$ can be written by listing its edges as $L = \{ s_1t_1, \ldots , s_Mt_M\}$,
with $s_l < t_l$ for each $l$.  For $M=1$, there is a unique lace and
$a=s_1<t_1=b$.   For $M \geq 2$, a graph is a lace
$L \in \Lcal^{(M)}[a,b]$
if and only if its edge endpoints can be ordered as
\begin{equation}
\label{e:LMij}
    a=s_1<s_2, \quad
    s_{l+1}<t_l\leq s_{l+2} \hspace{5mm} (l=1,\ldots,M-2), \quad
    s_M < t_{M-1} <t_M = b
\end{equation}
(for $M=2$ the middle inequalities are absent).  Thus $L$ divides $[a,b]$ into
$2M-1$ subintervals:
\begin{equation}
\label{e:LMint}
    [s_1,s_2], \; [s_2,t_1], \; [t_1,s_3], \; [s_3,t_2], \;
    \ldots \; ,[s_M,t_{M-1}], \; [t_{M-1},t_M].
\end{equation}
Of these, intervals number 3, 5, \ldots, $(2M-3)$ can have zero
length for $M \geq 3$,
whereas all others have length at least $1$.
This last fact will be important, as intervals which cannot have zero length
yield good factors for convergence of the lace expansion.

\begin{defn}
\label{def:lace_presc}
Given a connected graph $\Gamma$ on $[a,b]$,
the following prescription associates to $\Gamma$ a
lace ${\sf L}_\Gamma \subset \Gamma$:  The lace ${\sf L}_\Gamma$ consists
of edges $s_1 t_1, s_2 t_2, \ldots$, with $t_1,s_1,t_2,s_2, \ldots$
determined, in that order, by
\[
    t_1 = \max \{t : at \in \Gamma \} ,  \;\;\;\; s_1 = a,
\]
\[
    t_{i+1} = \max \{ t: \exists s < t_i \mbox{ such that }
    st \in \Gamma  \}, \;\;\;\;
    s_{i+1} = \min \{ s : st_{i+1} \in \Gamma \}  .
\]
The procedure terminates when $t_{i+1}=b$.
Given a lace $L$, the set of all edges $st \not\in L$ such that
${\sf L}_{L\cup \{st\} } = L $ is denoted  $\Ccal (L)$. Edges in
$\Ccal (L)$ are said to be \emph{compatible}
with $L$.
\end{defn}

Given a lace $L$ and the closed intervals \eqref{e:LMint} it determines, any edge $st$ with
each of $s,t$ lying in the same
one of those closed intervals is a compatible bond in $\Ccal(L)$.

\subsection{Definition of $\pi_m(x)$}
\label{sec-pidef}

For $m \ge 1$ and $x\in \Q^N$,
let $\Wcal_m(x)$ denote the set of all $m$-step walks
$\omega = (\omega(0),\omega(1),\ldots , \omega(m))$ on $\Q^N$
(possibly self-intersecting), with $|\omega(i)-\omega(i-1)|=1$
for $i =1,\ldots,m$, and with $\omega(0)=0$ and $\omega(m)=x$.
Given $\omega \in \Wcal_m(x)$, let
\begin{align}
\label{Ust}
    U_{st}(\omega) = \left\{ \begin{array}{rl}
    -1 & \mbox{if}\,\, \omega(s)=\omega(t)
    \\
    0 & \mbox{if}\,\,  \omega(s)\neq \omega(t).
    \end{array}
    \right.
\end{align}
Then
\begin{equation}
\label{clam}
    c_n (x) = \sum_{\omega \in \Wcal_n(x)} \prod_{0 \leq s<t \leq n}
    (1+ U_{st}(\omega)),
\end{equation}
since the product is equal to $1$
if $\omega$ is a self-avoiding walk and is equal to $0$ otherwise.
By expanding the product in \eqref{clam} we obtain
\begin{equation}
\label{e:cnexp}
    c_n (x) = \sum_{\omega \in \Wcal_n(x)}
    \sum_{\Gamma \in \Bcal[0,n]} \prod_{st \in \Gamma}
    U_{st}(\omega).
\end{equation}

In the sum over all graphs $\Gamma$ in \eqref{e:cnexp},  we partition
according to whether:
\[
\text{ (a)~~$0$ does not occur in an edge in $\Gamma$, \;\;\; or \;\;\;
(b)~~$0$ does occur in an edge in $\Gamma$.}
\]
This gives the identity
\eqref{e:cnrec1} (see \cite[p.~22]{Slad06} for details), namely
\begin{equation}
\label{e:cnrec}
    c_n(x) = N(D * c_{n-1})(x) + \sum_{m=2}^n (\pi_m * c_{n-m})(x),
\end{equation}
with, for $m \ge 2$,
\begin{equation}
\label{pimxconn}
    \pi_m(x) = \sum_{\omega \in \Wcal_m(x)}
    \sum_{\Gamma \in \Gcal[0,m]} \prod_{st \in \Gamma}
    U_{st}(\omega).
\end{equation}
Indeed, Case~(a) gives rise to the first term on the right-hand side of
\eqref{e:cnrec}, and Case~(b) gives rise to the second term with
$[0,m]$ the support of the connected component of $\Gamma$ containing $0$.

The sum over connected graphs can be reorganised by summing over laces $L$ and over
connected graphs for which the prescription of Definition~\ref{def:lace_presc}
produces $L$.  Then a resummation
of the sum over those connected graphs leads to the formula
\begin{equation}
\label{e:pimdef}
    \pi_m (x) =
        \sum_{\omega \in \Wcal_m( x)}
        \sum_{L \in \Lcal[0,m] } \prod_{st \in L}  U_{st}(\omega)
    \prod_{s't' \in \Ccal (L) } ( 1 + U_{s't'}(\omega) ).
\end{equation}
More details of this resummation can be found in
\cite[Section~3.3]{Slad06} or in either of \cite{BS85,MS93}.
The formula \eqref{e:pimdef} is the useful formula for
application of Proposition~\ref{prop:lace}.

A refinement of \eqref{e:pimdef} is obtained by
restricting the sum in \eqref{e:pimdef} to laces with $M$ edges, and we define
\begin{equation}
\label{e:PiMdef}
    \pi_m^{(M)} (x) =
        \sum_{\omega \in \Wcal_m( x)}
        \sum_{L \in \Lcal ^{(M)}[0,m]} \prod_{st \in L} (- U_{st}(\omega))
    \prod_{s't' \in \Ccal (L) } ( 1 + U_{s't'}(\omega) ).
\end{equation}
The minus sign has been introduced in the first product
of \eqref{e:PiMdef} in order to make $\pi_m^{(M)} (x)$
a nonnegative integer.
The right-hand side of \eqref{e:PiMdef} is zero unless $M < m$ (since
$U_{st}(\omega)=0$ if $t=s+1$ and the subset of $\Lcal^{(M)}[0,m]$ consisting
of laces with all edges of length at least two is empty if $M \ge m$),
and hence
\begin{equation}
\label{Pidef}
    \pi_m (x) = \sum_{M=1}^{m-1} (-1)^M \pi_m^{(M)} (x).
\end{equation}
Each term in the double sum \eqref{e:PiMdef} is either $0$ or $1$,
with the first product in \eqref{e:PiMdef} equal to $1$ if and only if
$\omega(s)=\omega(t)$ for each edge $st \in L$, while the second product is equal
to $1$ if and only if $\omega(s')\neq \omega(t')$ for each $s't' \in \Ccal(L)$.
Thus $\pi_m^{(M)}(x)$ counts the $m$-step ``lace graphs'' starting at the origin
and ending at $x$, with the specific
self-intersections that are enforced by the lace and with
the specific self-avoidance conditions
enforced by the compatible edges.
The required self-intersections are illustrated in Figure~\ref{fig:lacetop}.
For $M \ge 1$, the generating function of $\pi_m^{(M)}(x)$ is written as
\begin{equation}
    \Pi_z^{(M)}(x) = \sum_{m=2}^\infty \pi_m^{(M)}(x)z^m.
\end{equation}

\begin{figure}
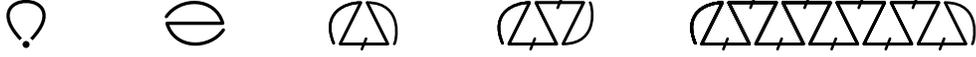

\begin{center}
\includegraphics[scale = 1.0]{figlacea-new.eps}
\includegraphics[scale = 1.0]{figlaceb-new.eps}
\includegraphics[scale = 1.0]{figlacec.eps}
\includegraphics[scale = 1.0]{figlaced.eps}
\includegraphics[scale = 1.0]{figlacee.eps}
\end{center}
\caption{
Self-intersections required for a
walk $\omega$ with $\prod_{st \in L}U_{st}(\omega) \neq 0$,
for the laces with $M=1,2,3,4$ edges depicted in Figure~\ref{fig:LM}.
The configuration for $M=11$ is also shown.  A slashed subwalk may have length zero,
whereas subwalks which are not slashed must take at least one step.}
\label{fig:lacetop}
\end{figure}

The simplest term is $\pi_m^{(1)}(x)$, which
is zero if $x \neq 0$.  Since every edge except $0m$ is compatible with
the unique $1$-edge lace $L=0m$, $\pi_m^{(1)}(x)$ is the number of $m$-step self-avoiding returns
to the origin when $x =0$.
Thus $\pi_m^{(1)}(0)$ is simply equal to $\sum_{i=1}^Nc_{m-1}(e_i)$, where here
the unit vector $e_i$ represents the penultimate vertex visited by the
self-avoiding return before it takes its final step to the
origin.
Its generating function is therefore
\begin{equation}
\label{e:Pi1-bis}
    \Pi_z^{(1)}(x) = \delta_{0,x} \sum_{m=2}^\infty \pi_m^{(1)}(x)z^m
    = \delta_{0,x} zN (D*G_z)(0).
\end{equation}

For $M \geq 2$, $\pi_m^{(M)}(x)$ counts $m$-step $M$-loop
walk configurations as indicated in Figure~\ref{fig:lacetop}.
The number of loops
in a diagram is equal to the number of edges in the corresponding lace.
Each of the $2M-1$ subwalks in a diagram is self-avoiding due to the compatible edges.
The compatible edges also enforce specific mutual avoidances between subwalks,
which can be neglected in upper bounds but which must be taken into account
to compute the coefficients $a_n$ in the asymptotic expansion of Theorem~\ref{thm:expansion}
(we return to this point in Section~\ref{sec:coefficients}).
The slashed lines in Figure~\ref{fig:lacetop} correspond to subwalks which may
consist of zero steps, but the others correspond to subwalks consisting
of at least one step (recall the discussion below \eqref{e:LMint}).

As an example of how to estimate $\Pi_z^{(M)}(x)$,
we consider the case $M=2$ in further detail.
A walk giving a contribution to $\pi_m^{(2)}(x)$ must travel from $0$ to $x$,
then back to $0$, and then finally return to $x$, as in the two-loop diagram in
Figure~\ref{fig:lacetop}.  Due to the product over compatible bonds in
\eqref{e:PiMdef}, each of these three subwalks must itself be self-avoiding, and $x$ cannot
equal $0$.  By relaxing the
avoidance between the three subwalks
we obtain an upper bound
\begin{equation}
\label{e:pi2}
    \pi_m^{(2)}(x) \le \sum_{\substack{m_1+m_2+m_3=m\\m_1,m_2,m_3 \ge 1}}
    c_{m_1}(x)c_{m_2}(x)c_{m_3}(x).
\end{equation}
We define the generating function
\begin{equation}
    H_z(x) = G_z(x) - \delta_{0,x} = \sum_{n=1}^\infty c_n^{(N)}(x) z^n
\end{equation}
for the sequence $c_n^{(N)}(x)$ with its  $n=0$ term omitted.
The generating function for $\pi_m^{(2)}(x)$ converts the convolution
in \eqref{e:pi2} into a product, so that
(since $\pi_m^{(2)}(x)=0$ when $x = 0$)
\begin{equation}
\label{e:Pi2H3}
    \Pi_z^{(2)}(x) \le  H_z(x)^3.
\end{equation}
We can then estimate the sum over $x$ of $\Pi_z^{(2)}(x)$ using
\begin{equation}
\label{e:Pi2a}
    \sum_{x\in\Q^N} \Pi_z^{(2)}(x) \le \sum_{x\in \Q^N}H_z(x)^3
    =
    (H_z*H_z^2)(0) \le (G_z*H_z^2)(0).
\end{equation}
This is the $M=2$ version of the inequality \eqref{e:PiMconv} that will appear below.
The right-hand side of \eqref{e:Pi2a} can be further estimated as
\begin{equation}
\label{e:Pi2b}
    \sum_{x\in\Q^N} \Pi_z^{(2)}(x) \le \|H_z\|_\infty (G_z*H_z)(0)
    \le \|H_z\|_\infty \|G_z * H_z\|_\infty,
\end{equation}
which is the $M=2$ version of the inequality \eqref{e:PizMbd} that will appear below.

\subsection{Diagrammatic estimates}
\label{sec:diag-est}

We define the multiplication operator $\Mcal_z$ and the convolution operator
$\Gcal_z$ by
\begin{align}
\label{e:Mcal}
    (\Mcal_z f)(x) & =  H_z(x)f(x),
    \\
\label{e:Gcal}
    (\Gcal_z f)(x) & =  (G_z *f)(x)
    ,
\end{align}
for $f: \Q^N \to \C$ and $x \in \Q^N$.
For such functions $f$,
we use the norms $\|f\|_\infty = \max_{x\in \Q^N}|f(x)|$ and
$\|f\|_p = [\sum_{x\in\Q^N}|f(x)|^p]^{1/p}$ for $p\in [1,\infty)$.

A proof of the following diagrammatic estimate can be found at \cite[(4.40)]{Slad06}.
For $\Pi_z^{(1)}$ it is \eqref{e:Pi1-bis}, since $G_z$ in \eqref{e:Pi1-bis} can
be replaced by $H_z$ since $D(0)=0$.
Although presented in \cite{Slad06} for $\Z^d$, the proof applies to the
hypercube mutatis mutandis.
Each of the $2M+1$ factors on the right-hand side of \eqref{e:PiMconv} arises
from one of the $2M+1$ lines in the $M$-loop diagrams depicted in Figure~\ref{fig:lacetop}.

\begin{prop}
\label{prop:PiM}
For $z \ge 0$,
\begin{equation}
\label{e:Pi1}
    \Pi^{(1)}_z(x)
    =
    \delta_{0,x} zN (D*H_z)(0),
\end{equation}
and for $M \ge 2$,
\begin{equation}
\label{e:PiMconv}
    \| \Pi^{(M)}_z\|_1
    \le
    \left[ (\Gcal_z \Mcal_z)^{M-1}  H_z\right] (0).
\end{equation}
\end{prop}

Note that estimates for $z \ge 0$ as in Proposition~\ref{prop:PiM} also imply bounds for
complex $z\in \C$ via
\begin{equation}
\label{e:PizC}
    |\Pi_z^{(M)}(x)| \le \Pi_{|z|}^{(M)}(x)
\end{equation}
since $\pi_m^{(M)}(x) \ge 0$.
The following lemma
provides a way to bound the right-hand side of \eqref{e:PiMconv}.
For its elementary proof see, e.g., \cite[Lemma~4.6]{Slad06};
the assumption there that the $f_i$ be even functions is vacuous for $\Q^N$ since
$x=-x$ for all $x\in\Q^N$.

\begin{lemma}
\label{lem:convol_bound}
Given nonnegative
functions $f_0,f_1,\ldots,f_{2q}$ on $\Q^N$, for $j=1,\ldots,q$ let $\Ccal_{j}$
and $\Mcal_{j}$ be the operators $(\Ccal_j f)(x) = (f_{2j}*f)(x)$ and
$(\Mcal_j f)(x) = f_{2j-1}(x)f(x)$.  Then for any
$k \in \{0,\ldots, 2q\}$,
\begin{equation}
\label{e:HMjboundaz}
    \|\Ccal_{q}\Mcal_{q}\cdots \Ccal_{1}\Mcal_{1} f_0\|_\infty
    \leq \|f_k\|_\infty \prod  \|f_{i}*f_{i'}\|_\infty ,
\end{equation}
where the product is over disjoint consecutive pairs
$ii'$ taken from the set $\{0,\ldots, 2q\}\setminus \{k\}$
(e.g., for $k=3$ and $q=3$, the product has factors with
$ii'$ equal to $01$, $24$, $56$).
\end{lemma}

Given a function $f:\Q^N\to \C$ and $k \in \Q^N$, we define $f_k:\Q^N \to \C$ by
\begin{equation}
\label{e:fkdef}
    f_k(x) = (1-(-1)^{k\cdot x})f(x).
\end{equation}
Also, given a
power series $f(z) = \sum_{n=0}^\infty a_n z^n$ and a real number $\epsilon > 0$,
we define the ``fractional derivative''
\begin{equation}
\label{e:depdef}
	\delta_z^\epsilon f(z) = \sum_{n=1}^\infty n^\epsilon a_n z^n.
\end{equation}
For $\epsilon$ equal to
a positive integer, $\delta_z^\epsilon$ does
not give the usual derivative but gives instead $( z \partial_z )^\epsilon$.

The following proposition gives norm estimates for $\Pi_z$, for $\Pi_{z,k}$
(defined by taking $f=\Pi_z$ in \eqref{e:fkdef}),
and for fractional $z$-derivatives of $\Pi_z$.
Its proof is a very minor modification of the proof of \cite[Theorem~4.1]{Slad06}
(which is inspired by \cite{BS85})
to which we refer the interested reader for the somewhat lengthy details.
Rather than repeating those details here,
we instead illustrate the ideas in the proof of \eqref{e:Pizder}--\eqref{e:Pikder}
by focussing on the cases $M=1,2$.

\begin{prop}
\label{prop:Pidots}
Let $z \ge 0$, $k\in\Q^N$, and $\epsilon \ge 1$.
For $M=1$, $\Pi_{z,k}^{(1)}(x)=0$ and
\begin{equation}
\label{e:Pi1bd}
    \|   \Pi^{(1)}_z \|_1
    \le zN\| H_z\|_\infty, \qquad
    \|  \delta_z^\epsilon   \Pi^{(1)}_z \|_1
    \le N\|\delta_z^\epsilon (zH_z)\|_\infty.
\end{equation}
For $M \ge 2$,
\begin{align}
\label{e:PizMbd}
    \|     \Pi^{(M)}_z \|_1
    &\le
    \| H_z\|_\infty
    \|H_z * G_z \|_\infty^{M-1}
    ,
\\
\label{e:Pizder}
    \|  \delta_z^\epsilon   \Pi^{(M)}_z \|_1
    &\le
    (2M-1)^{\epsilon}
    \|\delta_z^{\epsilon} H_z\|_\infty
    \|H_z * G_z \|_\infty^{M-1}
    ,
\\
\label{e:Pikder}
    \|\Pi^{(M)}_{z,k} \|_1
    &\le
    \floor{M/2}
    \|H_{z,k}\|_\infty
    \|H_z * G_z \|_\infty^{M-1}
    .
\end{align}
\end{prop}

\begin{proof}
Since $\Pi_z^{(1)}(x)=\delta_{0,x} zN(D*H_z)(0)$ by \eqref{e:Pi1},
it follows that $\Pi_{z,k}^{(1)}(x)=0$ as claimed, and also
the second bound of \eqref{e:Pi1bd} follows from
the identity
\begin{equation}
    \|  \delta_z^\epsilon   \Pi^{(1)}_z \|_1
    =
    \delta_z^\epsilon \big[zN(D*H_z)(0)\big]
    = N \sum_{i=1}^N N^{-1} \delta_z^\epsilon [zH_z(e_i)],
\end{equation}
where the $e_i$ are the unit vectors in $\Q^N$.
The first estimate of \eqref{e:Pi1bd} follows similarly, with $\delta_z^\epsilon$ omitted.

We restrict attention now to $M \ge 2$.  The bound \eqref{e:PizMbd}
is a consequence of \eqref{e:PiMconv} and Lemma~\ref{lem:convol_bound}
(with $k=0$), since
the right-hand side of \eqref{e:PiMconv}
is bounded by the left-hand side of \eqref{e:HMjboundaz} with $f_0=H_z$
and with $f_1,\ldots,f_{2(M-1)}$ alternating between $H_z$ and $G_z$.

For \eqref{e:Pizder}, by definition,
\begin{equation}
    \delta_z^\epsilon \Pi^{(M)}_z(x) = \sum_{m=2}^\infty m^\epsilon \pi_m^{(M)}(x) z^m.
\end{equation}
In the diagrammatic representation, $\pi_m^{(M)}(x)$ is represented by a diagram
with $2M-1$ subwalks of total length $m$.  Let $m_i$ be the length of the $i^{\rm th}$
subwalk.  By H\"older's inequality with exponents $\epsilon$ and $\frac{\epsilon}{\epsilon-1}$ (here is where the restriction
$\epsilon \ge 1$ is convenient),
\begin{equation}
\label{e:mHolder}
    m^\epsilon = \left( \sum_{i=1}^{2M-1} m_i\cdot 1 \right)^\epsilon
    \le (2M-1)^{\epsilon- 1} \sum_{i=1}^{2M-1} m_i^\epsilon.
\end{equation}
To see how this can be used in the simplest example, consider the case $M=2$.  In this case,
\eqref{e:pi2} and \eqref{e:mHolder} give
\begin{equation}
\label{e:pi2ep}
    m^\epsilon \pi_m^{(2)}(x)
    \le 3^{\epsilon-1}   \sum_{i=1}^3
    \sum_{\substack{m_1+m_2+m_3=m\\m_1,m_2,m_3 \ge 1}}
    m_i^\epsilon \;  c_{m_1}(x)c_{m_2}(x)c_{m_3}(x).
\end{equation}
We can bound the sum over $x$ of the
generating function of the left-hand side, term-by-term in the sum
over $i$ as in \eqref{e:Pi2a}, by
\begin{align}
\label{e:depPiH}
    \|  \delta_z^\epsilon   \Pi^{(2)}_z \|_1
    \le
    3^\epsilon \sum_{x\in \Q^N} \big(\delta_z^\epsilon H_z(x) \big) H_z(x)^2
    .
\end{align}
Observe that, along with the factor $3^\epsilon$,
one of the $H_z$ factors in \eqref{e:Pi2a} has now been replaced by
$\delta_z^\epsilon H_z$.  As in \eqref{e:Pi2a}--\eqref{e:Pi2b}, we can continue the estimate
with
\begin{align}
    \|  \delta_z^\epsilon   \Pi^{(2)}_z \|_1
    \le
    3^\epsilon \|\delta_z^\epsilon H_z\|_\infty \sum_{x\in \Q^N}H_z(x)^2
    \le
    3^\epsilon \|\delta_z^\epsilon H_z\|_\infty \|H_z*G_z\|_\infty
    ,
\end{align}
in agreement with \eqref{e:Pizder} for $M=2$.
For general $M\ge 3$, use of the inequality
\eqref{e:mHolder} leads to an upper bound for $\|\delta_z^\epsilon \Pi_z^{(M)}\|_1$ equal to
$(2M-1)^{\epsilon-1}$ times a sum of $2M-1$ terms with the $i^{\rm th}$ term
being the modification of \eqref{e:PiMconv} in which the $i^{\rm th}$ of the factors
$\Gcal_z,\Mcal_z,H_z$ has its function ($G_z$ or $H_z$) replaced by
$\delta_z^\epsilon H_z$ (note that $\delta_z^\epsilon G_z=\delta_z^\epsilon H_z$ by definition).
Consequently, with \eqref{e:HMjboundaz}
and choosing the modified factor as the distinguished
one in \eqref{e:HMjboundaz}, we see that
\begin{equation}
    \|\delta_z^\epsilon \Pi^{(M)}_z \|_1
    \le
    (2M-1)^{\epsilon} \|\delta_z^{\epsilon} H_z\|_\infty
    \|H_z * G_z \|_\infty^{M-1}
\end{equation}
as claimed.

Finally, for \eqref{e:Pikder},
we again illustrate this for the case $M=2$, as follows.  By definition,
and by \eqref{e:Pi2H3},
\begin{align}
        \|\Pi^{(2)}_{z,k} \|_1
        & =
        \sum_{x\in \Q^N}[1-(-1)^{k\cdot x}]\Pi_z^{(2)}(x)
        \nnb & \le
        \sum_{x\in \Q^N}\left([1-(-1)^{k\cdot x}]H_z(x)\right) H_z(x)^2
        =
        \sum_{x\in \Q^N} H_{z,k}(x)  H_z(x)^2.
\end{align}
This is reminiscent of \eqref{e:depPiH}, with the difference that one factor
on the right-hand side is $H_{z,k}(x)$ rather than $\delta_z^\epsilon H_z(x)$.
By using the supremum norm on that factor, we can similarly obtain an upper bound
\begin{equation}
        \|\Pi^{(2)}_{z,k} \|_1
        \le
        \|H_{z,k}\|_\infty \|H_z*G_z\|_\infty,
\end{equation}
which is the $M=2$ case of \eqref{e:Pikder}.
For general $M\ge 3$, the proof is a very small adaptation of the proof of
\cite[(4.10)]{Slad06}.
We divide the displacement $x$ in $\pi_m(x)$ as
a sum $x=\sum_{i=1}^{\floor{M/2}}x_i$ over displacements $x_i$ along the subwalks along
the bottom of the $M$-loop diagram depicted in Figure~\ref{fig:lacetop}.  We use the inequality
\begin{equation}
    1-(-1)^{k\cdot x} \le \sum_{i=1}^{\floor{M/2}}[1-(-1)^{k\cdot x_i}]
\end{equation}
which holds if $k\cdot x$ is even since the left-hand side is then zero, and holds
if $k\cdot x$ is odd since then at least one of the $k\cdot x_i$ must also be odd.
Use of this inequality leads to
an upper bound for $\|\Pi_{z,k}^{(M)}\|_1$ consisting of a sum of $\lfloor M/2\rfloor$
terms, each of which is the modification of \eqref{e:PiMconv} in which one of the
factors $\Gcal_z,\Mcal_z,H_z$, has its function $(G_z$ or $H_z$) replaced by
$H_{z,k}$ (note that $H_{z,k}=G_{z,k}$ by definition).  With \eqref{e:HMjboundaz}
and with the modified factor chosen as the distinguished
one in \eqref{e:HMjboundaz}, this leads to \eqref{e:Pikder}.
\end{proof}

\section{Random walk on the hypercube}
\label{sec:rw}

The convergence proof for the lace expansion makes use of a comparison with
simple random walk on the hypercube.  In this section, we prove the two estimates
needed for that task, in Lemma~\ref{lem:Dsum}.

Recall from \eqref{e:Ddef} that
$D(x) = N^{-1} \delta_{|x|,1}$ is the transition probability for simple random walk
on the hypercube.
Its Fourier transform is given in \eqref{e:Dhat} as
$\hat D(k) = 1 - \frac{2|k|}{N}$.
The following lemma, which is similar to but simpler
than what appears in \cite[Section~2]{BCHSS05b}
due to our restriction to the hypercube,
provides essential estimates for
the convergence proof for the lace expansion in Section~\ref{sec:le-conv}.

\begin{lemma}
\label{lem:Dsum}
For $i \ge 0$ there is a constant $c_i$ such that
\begin{equation}
\label{e:tranprob}
    \max_{x \in \Q^N}
    \frac{1}{V} \sum_{k \in \Q^N} \hat D(k)^i (-1)^{k \cdot x}
    \le c_{i} N^{-\lceil i/2\rceil}.
\end{equation}
For $i,j \ge 0$ there is a constant $c_{i,j}$ such that for all $t \in [0,1]$
\begin{equation}
\label{e:Dijsum}
    \frac{1}{V} \sum_{k \in \Q^N: k \neq 0} \frac{|\hat D(k)^i|}{[1-t\hat D(k)]^j}
    \le
    c_{i,j} N^{-i/2}.
\end{equation}
\end{lemma}

\begin{proof}
By inverse Fourier transformation, the normalised sum in \eqref{e:tranprob}
is the transition probability for simple random walk to travel from
$0$ to $x$ in $i$ steps:
\begin{equation}
    D^{*i}(x)  = \frac{1}{V} \sum_{k \in \Q^N} \hat D(k)^i (-1)^{k \cdot x}
\end{equation}
and hence it is nonnegative and equals zero if $i$ and $|x|$ have different parity.
It is equal to $\delta_{0,x}$ for $i=0$ so we may assume that $i \ge 1$.
Closely related explicit transition probabilities are written in terms of
Krawtchouk polynomials in \cite{DGM90} but we can instead proceed crudely here
with an elementary counting argument.
 Without loss of generality we may assume by symmetry
that $x$ consists of a string of $|x|$ 1's followed by $(N-|x|)$ 0's.
There are $N^{i}$ possible $i$-step walks starting from $0$.  The number of those that end
at $x$ can be bounded as follows.
First we observe that the first $|x|$ coordinates of $x$ must flip an odd number
of times (each at least once),
whereas the remaining coordinates must flip an even number of times (possibly zero).
Let $\delta$ be the total number of coordinates that do flip.  Since the first $|x|$ coordinates
flip at least once and the remaining $\delta-|x|$ flip at least twice, it must be the case
that $|x|+ 2(\delta-|x|) \le i$, which implies that $\delta -|x| \le \frac 12 (i-|x|)$.
The number of ways to choose which of the $N-|x|$ coordinates are the $\delta-|x|$
coordinates that flip a positive even number of times is at most
$(N-|x|)^{\delta-|x|} \le N^{(i-|x|)/2}$.  Since the number of $i$-step walks that
flip $\delta$ specific coordinates is $\delta^i$,
we find that the transition probability obeys the inequality
\begin{align}
    D^{*i}(x) & \le \frac{1}{N^i}\sum_{\delta=|x|}^i N^{(i -|x|)/2} \delta^i
    \le \frac{1}{N^{(i +|x|)/2}}  (i+1) i^i.
\end{align}
When $i$ is even the factor $N^{-(i +|x|)/2}$ on the right-hand side is
at most $N^{-i/2}$ (since $|x|\ge 0$), whereas for $i$ odd it is at
most $N^{-(i+1)/2}$ (since $|x|$ must also be odd so at least $1$).  This completes the proof of \eqref{e:tranprob}.

Next we consider \eqref{e:Dijsum}.
By the Cauchy--Schwarz inequality,
\begin{equation}
\label{e:DijCS}
    \frac{1}{V} \sum_{k \in \Q^N: k \neq 0} \frac{|\hat D(k)^i|}{[1-t\hat D(k)]^j}
    \le
    \left[\frac{1}{V} \sum_{k \in \Q^N: k \neq 0} \hat D(k)^{2i}\right]^{1/2}
    \left[\frac{1}{V} \sum_{k \in \Q^N: k \neq 0} \frac{1}{[1-t\hat D(k)]^{2j}}\right]^{1/2}
    .
\end{equation}
The first factor on the right-hand side is at most a multiple of $N^{-i/2}$ by
\eqref{e:tranprob} (applied just for $x=0$), so it suffices to prove that for any $j \ge 1$
(the case $j =0$ is clear)
\begin{align}
    \frac{1}{V} \sum_{k \neq 0} \frac{1}{[1-t\hat D(k)]^j} &\le c_{j}
\end{align}
for some positive $c_j$.
Since $\hat D(k) \in [-1,1]$ the left-hand side is bounded above by $2^j$ if
$t \in [0,\frac 12]$.  For the more substantial case of
$t \in [\frac 12,1]$ we have
\begin{equation}
    1-t\hat D(k) = 1-t + \frac{2t|k|}{N} \ge \frac{|k|}{N}
\end{equation}
and therefore in this case
\begin{align}
     \frac{1}{V} \sum_{k \neq 0} \frac{1}{[1-t\hat D(k)]^j}
    &\le
    N^j \frac{1}{V}\sum_{k \neq 0} \frac{1}{|k|^j}
    =
    N^j \frac{1}{V}\sum_{m=1}^N \binom{N}{m} \frac{1}{m^j}
    .
\end{align}
We divide the sum over $m$ according to whether $m \le \frac 14 N$ or $m> \frac 14 N$.
For the second case we use
\begin{align}
    N^j \frac{1}{V}\sum_{m=N/4}^{N} \binom{N}{m} \frac{1}{m^j}
    & \le
    N^j \frac{4^j}{N^j} = 4^j
    .
\end{align}
For the first case, with $X_N$ a random variable with Bin($N,\frac 12$) distribution, we have
\begin{align}
    N^j \frac{1}{V}\sum_{m=1}^{N/4} \binom{N}{m} \frac{1}{m^j}
    & \le
    N^j \frac{1}{V}\sum_{m=0}^{N/4} \binom{N}{m}
    =
    N^j \P(X_N \le N/4)
    \le
    N^j e^{-N/8}
    ,
\end{align}
where we used the Chernoff bound
$\P(X_N \le a) \le \exp[-(N-2a)^2/(2N)]$ for $2a<N$
in the last step (see, e.g, \cite{AS00}).  Since the right-hand
side is bounded by a $j$-dependent constant, the proof is complete.
\end{proof}

\section{Convergence of the lace expansion}
\label{sec:le-conv}

In this section we prove the convergence of the lace expansion,
using the strategy of \cite[Section~5.2]{Slad06} which itself is based on the
strategy used in \cite{BCHSS05b}.
Like most lace expansion convergence proofs, we use a bootstrap argument.
The bootstrap argument is presented in Section~\ref{sec:bootstrap}.
The fact that we are working on the hypercube
makes for considerable simplification and this convergence proof is simpler
than that in \cite{BCHSS05b,Slad06} (for a different kind of simplification see
\cite{Slad22_lace} for weakly self-avoiding walk on $\Z^d$ for $d>4$).

\subsection{Preparation}

We recall from \eqref{e:fkdef} the definition
\begin{equation}
\label{e:fkdef-bis}
    f_k(x) = (1-(-1)^{k\cdot x})f(x).
\end{equation}
In particular, $f_0(x)=0$ for all $f$.
The Fourier transform of $f_k$ is
\begin{equation}
\label{e:fkFT}
    \hat f_k(\ell) = \sum_{x\in \Q^N} (1-(-1)^{k\cdot x})f(x) (-1)^{\ell\cdot x}
    =
    \hat f(\ell) - \hat f(k+\ell)
    \qquad
    (k,\ell \in \Q^N).
\end{equation}
In particular, $\hat f_k(0)=\hat f(0)-\hat f(k)$, and we will use bounds on
$\hat f_k(\ell)$ to control differences of this type (see
\eqref{e:fkl1}--\eqref{e:fkl2} and \eqref{e:fkl3}--\eqref{e:fkl4}).

For $x \in \Q^N$, let $w_n(x)$ denote the number of $n$-step walks (not necessarily
self-avoiding) from $0$ to $x$.  Then $w_n(x) = N^n D^{*n}(x)$
and $\hat w_n(k) = [N\hat D(k)]^n$.  For $p \in [0,1/N)$
we define the generating function
\begin{equation}
\label{e:Cpdef}
    C_p(x) = \sum_{n=0}^\infty w_n(x)p^n.
\end{equation}
The Fourier transform of $C_p$ is (recall \eqref{e:Dhat})
\begin{equation}
\label{e:Cphatdef}
    \hat C_p(k) = \frac{1}{1-pN\hat D(k)} = \frac{1}{1-pN + 2p|k|}.
\end{equation}
If we evaluate the above right-hand side at $p=1/N$ then it becomes
$N/(2|k|)$, and
the \emph{zero mode} (namely the case $k=0$) is divergent.
This is a symptom of the recurrence of simple random walk
on the hypercube.  The zero mode
was excluded in  \eqref{e:Dijsum}
where the denominator is zero for $k=0$ if $t=1$.
The zero mode will play an important role in the bootstrap argument and also subsequently
in Section~\ref{sec:pfF}.

For later use, we observe that for $p \in [0,1/N)$ it follows from
\eqref{e:fkFT} with $f= C_p$ and from \eqref{e:Cphatdef} that
\begin{align}
    \hat C_{p,k}(\ell) &= pN[\hat D(\ell) - \hat D(k+\ell)]\hat C_p(\ell) \hat C_p(k+\ell)
    \nnb & \le
    [1-\hat D(k)] \hat C_p(\ell) \hat C_p(k+\ell)
    \nnb &
    = \bar C_p(k,\ell),
\label{e:Cbardef}
\end{align}
where the last equality defines $\bar C_p(k,\ell)$.

\subsection{The bootstrap argument}
\label{sec:bootstrap}

The following lemma is the basis for the bootstrap argument.

\begin{lemma}
\label{lem:P4}
Let $a<b$, let $f$ be a continuous
function on the interval $[z_1,z_2]$, and assume that $f(z_1) \leq a$.
Suppose for each $z \in (z_1,z_2]$ that if $f(z) \leq b$
    then in fact $f(z) \leq a$.
Then $f(z) \leq a$ for all $z \in [z_1,z_2]$.
\end{lemma}

\begin{proof}
By hypothesis, $f(z)$ cannot lie
in the interval $(a,b]$  for
any $z \in (z_1,z_2]$.
Since $f(z_1) \leq a$, it follows by continuity that $f(z) \leq a$
for all $z \in [z_1,z_2]$.
\end{proof}

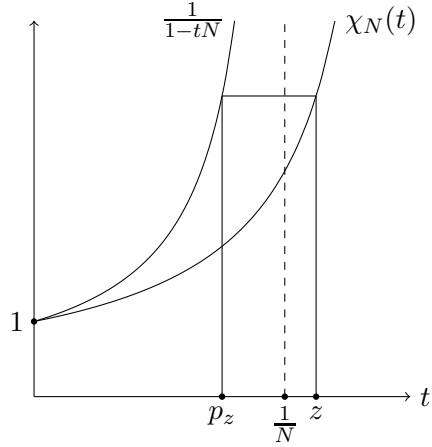
\begin{figure}[t]

\centering
\begin{tikzpicture}[xscale=5/3, yscale=1]
\draw[->] (0,0) -- (3,0) node[right] {$t$};
\draw[->] (0,0) -- (0,5) node[above] {};

\draw[domain =0:1.6, smooth] plot (\x ,{2/(2-\x)}) node[left] {$\frac{1}{1-tN}$};
\draw[domain =0:2.4, smooth] plot (\x ,{3/(3-\x)}) node[right] {$\chi_N(t)$};

\draw[-] (2.25, 0) -- (2.25, 4);
\draw[-] (2.25, 4) -- (1.5, 4);
\draw[-] (1.5, 4)   -- (1.5, 0);
\draw[dashed] (2, 0) -- (2,5);

\filldraw[black] (0,1) ellipse (0.6pt and 1pt) node[anchor=east] { $1$};
\filldraw[black] (2.25,0) ellipse (0.6pt and 1pt) node[anchor=north] { $z$};
\filldraw[black] (2,0)      ellipse (0.6pt and 1pt) node[anchor=north] { $\tfrac{1}{N}$};
\filldraw[black] (1.5,0)   ellipse (0.6pt and 1pt) node[anchor=north] { $p_z$};

\end{tikzpicture}
  \caption{The definition of $p_z$, illustrated for $z> 1/N$.}
  \label{fig:pofz}
\end{figure}

For $z \in [0,\infty)$, we define $p_z\in [0,1/N)$ as in Figure~\ref{fig:pofz} by
\begin{equation}
    \hat G_z(0)= \chi_N(z) = \frac{1}{1-p_zN} = \hat C_{p_z}(0).
\end{equation}
Equivalently, from \eqref{e:chiPi} we see that
\begin{equation}
\label{e:pzdef}
    p_zN = 1- \frac{1}{\chi_N(z)}
    =
    1-F_N(z) = zN +\hat\Pi_z(0).
\end{equation}
Our choice of $f$ in Lemma~\ref{lem:P4} is motivated by the intuition that
$\hat G_z(k)$ and $\hat C_{p_z}(k)$ are comparable in size, not just for $k=0$ where
they are equal by definition but also for all $k \in \Q^N$.
We also anticipate that $\hat G_{z,k}(\ell)$ and $\hat C_{p_z,k}(\ell)$
should be comparable, but
it is convenient and also sufficient to compare instead  $\hat G_{z,k}(\ell)$ and
the upper bound
$\bar C_{p_z}(k,\ell)$ for $\hat C_{p_z,k}(\ell)$ from \eqref{e:Cbardef}.

We will apply Lemma~\ref{lem:P4} with $z_1=0$,
$z_2 =z_N$, $a=2$, $b=4$,
and
\begin{equation}
    \label{e:fdef}
    f(z) = \max\{f_1(z),f_2(z), f_3(z)\},
\end{equation}
where
\begin{equation}
    f_1(z)=zN,
    \qquad
    f_2(z) =
    \max_{k\in \Q^N}
    \frac{|\hat G_z(k)|}{\hat{C}_{p_z}(k)},
    \qquad
\label{e:f3def}
    f_3(z) =
    \max_{k \in \Q^N\setminus \{0\}}
    \max_{l \in \Q^N}
    \frac{|\hat G_{z,k}(\ell)|}
    {\bar C_{p_z}(k,\ell)}.
\end{equation}
The omission of $k=0$ in the definition of $f_3$ avoids the ratio $\frac 00$.
By definition $p_0=0$, and since $\hat G_0(k)=\hat C_0(k)=1$, it follows that
$f_1(0)=0$, $f_2(0)=1$, $f_3(0)=0$ and hence $f(0)=1\le 2$.
The continuity of $f$ on $[0,z_N]$ also
follows easily from the continuity of $z \mapsto p_z$ and the
continuity (for fixed $k$) of $\hat G_z(k)$ and $\hat C_p(k)$ in $z$ and $p$.
We will verify that the remaining and substantial bootstrap
hypotheses of Lemma~\ref{lem:P4} holds
when $z_N=z_N(\lambda)$ is defined with a $\lambda$ that is sufficiently small.
From this, we can conclude that $f(z) \leq a=2$
uniformly in $z \in [0,z_N]$.

Once we conclude that $f(z) \le 2$ we
of course also know that $f(z) \le 4$
and hence all the conclusions of
Lemmas~\ref{lem:TDa}, \ref{lem:Pibds}, \ref{lem:43} and Remark~\ref{rk:bubble},
which are initially conditional on $f(z) \le 4$,
in fact hold unconditionally.
One of those conclusions, in the proof of Lemma~\ref{lem:43}, is that the ratio
$\hat G_z(k)/\hat{C}_{p_z}(k)$ is close to $1$ \emph{without} the absolute value taken
in the definition of $f_2$.
In particular $\hat G_z(k) \ge 0$; absolute values are included in $f_2$
since this is not obvious a priori.
The upper bound $\hat G_z(k)\le 2\hat{C}_{p_z}(k)$ is an example of what
is known as an \emph{infrared bound}, so named to place emphasis on its
significance for $k$ near zero (low frequency).

We define
\begin{equation}
\label{e:betazdef}
    \beta_z = \frac 1N + \frac{\chi_N(z)^2}{V} .
\end{equation}
If we assume that $z\le z_N$ with $z_N=z_N(\lambda)$ defined
by some $\lambda>0$ in \eqref{e:zNdef},
then for all $z \in [0,z_N]$ we have $\beta_z
\le N^{-1} + \lambda^2$.  Thus we can use $\beta_z$ as a small parameter,
assuming (as we will) that
$N^{-1} + \lambda^2$ is indeed small by demanding that $\lambda\in (0,\lambda_0]$
for sufficiently small $\lambda_0$.

For $p \in [1,\infty)$, we use the norms
$\|\hat f \|_{\hat p} = [V^{-1}\sum_{k \in \Q^N}|\hat f(k)|^p ]^{1/p}$ for the Fourier transform
as well as $\|f\|_p = [\sum_{x \in \Q^N}| f(x)|^p ]^{1/p}$ for untransformed functions,
so it is necessary to notice hats with norms to distinguish between the presence
or not of the volume factor.  We also use $\|f\|_\infty = \max_{x\in\Q^N}|f(x)|$.
The bound $\|f\|_\infty \leq \|\hat f\|_{\hat 1}$ follows from \eqref{e:IFT}.
The Parseval relation asserts that $\|f\|_2=\|\hat f\|_{\hat 2}$.
The convolution $(f*g)(x) = \sum_{y\in\Q^N}f(x-y)g(y)$ obeys
$\|f*g\|_\infty \leq \|f\|_2\|g\|_2$ by the Cauchy--Schwarz inequality,
and $\widehat{f*g}=\hat f \hat g$.

\begin{lemma}
\label{lem:TDa}
Fix $z \in (0,z_N]$ and assume that $f$ of \eqref{e:fdef} obeys
$ f(z) \leq K$.
Then there is a constant $c_K$, independent of $z$, such that
\begin{equation}
\label{e:Hbd}
   \|H_{z,k}\|_\infty \leq c_K (1+\lambda^2)  [ 1 -  \hat D(k) ] ,
    \qquad
    \| H_z\|_2^2 \leq c_K \beta_z ,
    \qquad
     \| H_z\|_\infty \leq c_K \beta_z .
\end{equation}
\end{lemma}

\begin{proof}
Special attention is required for the zero mode, which is
the origin of the term $V^{-1}\chi_N(z)^2$ in $\beta_z$.

Since $f_3(z) \le K$, it follows from the
definition of $\bar C_p(k,\ell)$ in \eqref{e:Cbardef}
and  the Cauchy--Schwarz inequality that
\begin{align}
   \|H_{z,k}\|_\infty
   =
   \|G_{z,k}\|_\infty
   \le \|\hat G_{z,k}\|_{\hat 1}
   \le K \|\bar C_{p_z}(k,\cdot)\|_{\hat 1}
    & \le  K [1- \hat D(k)]\|\hat C_{p_z}\|_{\hat 2}^2 .
\end{align}
Now we apply the definition of $p_z$, Lemma~\ref{lem:Dsum},
and the fact that $z\le z_N$ (so $\chi_N(z) \le \lambda V^{1/2}$) to see that
\begin{align}
    \|\hat C_{p_z}\|_{\hat 2}^2
    & = \frac 1V \sum_{k\in\Q^N} \hat C_{p_z}(k)^2
    = \frac 1V \hat C_{p_z}(0)^2 + \frac 1V \sum_{k\neq 0} \frac{1}{(1-p_zN\hat D(k))^2}
    \nnb & =
    \frac {\chi_N(z)^2}{V} + O(1) \le \lambda^2 + O(1) .
\end{align}
This proves the first bound of \eqref{e:Hbd}.

To estimate $\| H_z\|_2^2$,
we first use submultiplicativity in the form of the inequality
$c_n(x) \le (c_1*c_{n-1})(x) = N(D*c_{n-1})(x)$, along with $f_1(z) \leq K$, to obtain
\begin{equation}
\label{e:HDG}
    H_z(x) \leq zN( D*G_z)(x) \leq K (D*G_z)(x).
\end{equation}
With the Parseval relation and $f_2(z) \leq K$, this implies that
\begin{align}
    \| H_z\|_2^2 & \leq  K^2 \|D*G_z\|_2^2 = K^2 \|\hat D  \hat G_z\|_{\hat 2}^2
    \leq  K^4 \|\hat D  \hat C_{p_z}\|_{\hat 2}^2
     .
\end{align}
Then we estimate the right-hand side by extracting the zero mode and
using Lemma~\ref{lem:Dsum} for the nonzero $k$, as we did above.  This gives
\begin{align}
    \|\hat D \hat C_{p_z}\|_{\hat 2}^2
    &
    = \frac 1V \hat C_{p_z}(0)^2 + \frac 1V \sum_{k\neq 0} \frac{\hat D(k)^2}{(1-p_zN\hat D(k))^2}
    =
    \frac {\chi_N(z)^2}{V} + O(N^{-1})
    \le
    O(\beta_z)
    ,
\end{align}
which proves the second bound of \eqref{e:Hbd}, for a suitable constant $c_K$.

Iteration of \eqref{e:HDG}
using $G_z(x)=\delta_{0,x}+H_z(x)$ gives $H_z(x) \leq KD(x) + K^2 (D*D*G_z)(x)$.
Therefore,
\begin{align}
    \|H_z\|_\infty
    & \leq
    K\|D\|_\infty + K^2 \|\hat{D}^2 \hat{G}_z\|_{\hat 1}
    \leq
    K N^{-1}  + K^3  \|\hat{D}^2 \hat C_{p_z}\|_{\hat 1}
    ,
\end{align}
where we used $D(x) \le N^{-1}$ and
our assumption $f_2(z) \le K$ to bound $\|\hat{D}^2 \hat{G}_z\|_{\hat 1}$.
The norm on the right-hand side is equal to
\begin{align}
    \frac{\chi_N(z)}{V} + \frac 1V \sum_{k \neq 0} \frac{\hat D(k)^2}{1-p_zN\hat{D}(k)}
    \leq
    \frac{\chi_N(z)^2}{V} + O(N^{-1}) \le O(\beta_z),
\end{align}
where we used the inequality \eqref{e:Dijsum} of Lemma~\ref{lem:Dsum}
(with $i=2$ and $j=1$)
to bound the sum in the last line.
This proves the third bound of \eqref{e:Hbd}.
\end{proof}

\begin{rk}
\label{rk:bubble}
As mentioned below \eqref{e:f3def}, once the bootstrap proof is complete, statements
that follow from the assumption $f(z) \le K$ with $K=4$ in fact will then
be known to hold unconditionally with $K=2$.  In particular, we can conclude from
\eqref{e:Hbd}
(together with the fact that $H_z(0)=0$ by definition)
that for any $z\in [0,z_N]$ the \emph{bubble diagram}
\begin{equation}
    \|G_{z}\|_2^2 = 1  + \|H_z\|_2^2
\end{equation}
is bounded above by $1+O(\beta_z)$.  When $z=z_N$, this is a statement of
a \emph{bubble condition} analogous to the triangle condition for percolation
on a finite graph studied in \cite{BCHSS05a}.
\end{rk}

\begin{lemma}
\label{lem:Pibds}
Fix $z \in (0,z_N]$, and suppose that $f$ of \eqref{e:fdef} obeys
$ f(z) \leq K$.
Then there is a constant $\bar c_K$
(independent of $z$) such that if $\lambda \in (0,\lambda_0]$ with $\lambda_0$
sufficiently small (independent of $z$) then
\begin{equation}
\label{e:diffPibd}
    \| \Pi_z \|_1
    \leq \bar{c}_K  \beta_z,
    \qquad
    \| \Pi_{z,k}\|_1
    \leq \bar{c}_K  \beta_z [1- \hat D(k)].
\end{equation}
\end{lemma}

\begin{proof}
From Proposition~\ref{prop:Pidots} we have that
$\Pi_{z,k}^{(1)}(x)=0$ and
\begin{equation}
\label{e:2Pi1bd}
    \|\Pi^{(1)}_z\|_1     \le zN\| H_z \|_\infty,
\end{equation}
and, for $M \ge 2$, also that
\begin{align}
\label{e:2Pizder}
    \| \Pi^{(M)}_z \|_1
    &\le
    \| H_z\|_\infty
    \|H_z * G_z \|_\infty^{M-1}
    ,
\\
\label{e:2Pikder}
    \| \Pi^{(M)}_{z,k}\|_1
    &\le
    \floor{M/2}
    \|H_{z,k}\|_\infty
    \|H_z * G_z \|_\infty^{M-1}
    .
\end{align}
Since $H_z*G_z = H_z + (H_z*H_z)$ by definition,
the Cauchy--Schwarz inequality and Lemma~\ref{lem:TDa} give
\begin{equation}
\label{e:HstarG}
    \|H_z * G_z \|_\infty \le \|H_z\|_\infty + \|H_z*H_z\|_\infty
    \le \|H_z\|_\infty + \|H_z\|_2^2 \le 2c_K \beta_z .
\end{equation}
Since $zN \le K$ by assumption,
\begin{align}
    \|\Pi_z\|_1 \le \sum_{M=1}^\infty \|\Pi_z^{(M)}\|_1 \le \bar c_K \beta_z.
\end{align}
For the second bound of \eqref{e:diffPibd}, we similarly use
\begin{align}
    \|\Pi_{z,k}\|_1
    & \le
    \sum_{M=2}^\infty \lfloor M/2 \rfloor \|H_{z,k}\|_\infty \|H_z * G_z \|_\infty^{M-1}
    \nnb & \le
    \sum_{M=2}^\infty \lfloor M/2 \rfloor c_K(1+\lambda^2)[1-\hat{D}(k)]
    (2c_K \beta_z)^{M-1}
    \le
    \bar{c}_K  \beta_z [1- \hat D(k)].
\end{align}
Here we have taken $\lambda_0$ sufficiently small to control the geometric sum over $M$
(and we always consider large $N$).
\end{proof}

The next lemma completes the bootstrap argument by establishing the substantial
hypothesis of Lemma~\ref{lem:P4} for small $\lambda_0$.

\begin{lemma}
\label{lem:43}
Fix $z \in (0,z_N]$ and suppose that $f(z) \leq 4$.
For $\lambda\in (0,\lambda_0]$ with $\lambda_0$ sufficiently small
(independent of $z$),  it is in fact the case that $f(z) \leq 1+c\beta_z$
for some $c>0$ independent of $z$.
\end{lemma}

\begin{proof}
We consider $f_1,f_2,f_3$ in that order.

\smallskip\noindent \emph{Bound on $f_1(z)$.}
For $f_1(z)$, we simply note that $\chi_N(z) >0$ and hence also
\begin{equation}
    \chi_N(z)^{-1} = 1-zN -\hat\Pi_z(0) >0.
\end{equation}
Therefore, by Lemma~\ref{lem:Pibds} and the fact that
any function $h$ obeys $|\hat h(k)|\le \|h\|_1$ for all
$k$ (including the present case of $k=0$),
\begin{equation}
    f_1(z) = z N
    <
    1-\hat{\Pi}_z(0) \leq 1+\bar c_4\beta_z,
\end{equation}
assuming $\lambda_0$ is sufficiently small.

\smallskip\noindent \emph{Bound on $f_2(z)$.}
For $f_2$, we first recall \eqref{e:GFk} and write
\begin{equation}
\label{e:Fzhatdef}
    \hat{F}_z(k) = \frac{1}{\hat G_z(k)} = 1-zN \hat D(k) - \hat \Pi_z(k)
\end{equation}
(this gives an alternate notation $\hat F_z(0)$ for the reciprocal $F_N(z)$
of the susceptibility), so that
\begin{equation}
\label{e:GC1}
    \frac{\hat G_z(k)}{\hat C_{p_z}(k)}
    =
    \frac{1-p_zN\hat{D}(k)}{\hat F_z(k)}
    =
    1 + \hat E_z(k), \qquad \hat E_z(k) = \frac{1-p_zN\hat{D}(k)-\hat F_z(k)}{\hat F_z(k)}.
\end{equation}
We will show that $\hat E_z(k)=O(\beta_z)$,
which implies that $f_2(z) = 1+O(\beta_z)$.
By \eqref{e:pzdef},
 $p_zN = 1-\hat{F}_z(0) = zN +\hat\Pi_z(0)$, and thus by \eqref{e:GFk}
the numerator of $\hat E_z(k)$ is
\begin{align}
    1-p_zN\hat{D}(k)-\hat F_z(k)
    & =
    -\hat\Pi_z(0)\hat{D}(k) + \hat\Pi_z(k)
    = \hat\Pi_z(0)[1-\hat D(k)] - \hat\Pi_{z,k}(0).
\end{align}
We can now use our bound on $\hat\Pi_{z,k}(0)=\hat\Pi_z(0)-\hat\Pi_z(k)$.
Indeed, by \eqref{e:diffPibd}
 (again with $|\hat h(k)| \le \|h\|_1$)
\begin{align}
\label{e:fkl1}
    |\hat E_z(k)|
    &
    \le 2 \bar{c}_4\beta_z \frac{1-\hat{D}(k)}{|\hat F_z(k)|}.
\end{align}
For $z \leq \frac{1}{2N}$, we
can use the crude bound $c_n^{(N)}\le N^n$ to see that $|\hat G_z(k)| \le
\chi_N(z) \le \hat C_z(0) \le 2$ and
hence that
$|\hat E_z(k) | \le 8\bar c_4\beta_z$.
For $\frac{1}{2N} \leq z \le z_N$, we use
\begin{align}
    |\hat F_z(k)| & =  |\hat F_z(0) + [\hat F_z(k) - \hat F_z(0)]|
    \nnb
    & =  |\hat F_z(0) + zN[1-\hat D(k)] + \hat \Pi_{z,k}(0) |
    \nnb &\geq \hat F_z(0) + \frac 12[1-\hat{D}(k)]
   - \bar{c}_4 \beta_z[1-  \hat{D}(k)]
   \nnb & \geq
   \frac 14   [1- \hat{D}(k)] .
\label{e:fkl2}
\end{align}
Therefore  $|\hat E_z(k)| \le 16\bar c_4\beta_z$ (for small $\lambda_0$), and we
have proved that $f_2(z) = 1+O(\beta_z)$.

\smallskip\noindent \emph{Bound on $f_3(z)$.}
Although we only need a bound for $k\neq 0$, the following applies in fact for all $k \in \Q^N$.
We write
\begin{equation}
\label{e:gdef}
    \hat g_z(k) = zN \hat D(k) + \hat\Pi_z(k),
\end{equation}
so that
\begin{equation}
    \hat{G}_z(k) = \frac 1{1-\hat{g}_z(k)}.
    \label{e:Fdef}
\end{equation}
By \eqref{e:fkFT}, for all $k,\ell \in \Q^N$,
\begin{align}
    |\hat G_{z,k}(\ell)|
    & =
    |\hat G_z(\ell)|\, | \hat G_z(k+\ell)|\,  \left| \hat g_z(\ell) - \hat g_z(k+\ell) \right|
    \nnb & \le
    |\hat G_z(\ell)|\, |\hat G_z(k+\ell) |\,  \sum_{x\in\Q^N} [1-(-1)^{k\cdot x}] |g_z(x) |,
\end{align}
where
\begin{equation}
    g_z(x) = zN D(x) + \Pi_z(x)
\end{equation}
is the inverse Fourier transform of $\hat g_z(k)$ (recall \eqref{e:IFT}).
Since $f_2(z) \leq 1 + O(\beta_z)$, we can bound each factor
of $|\hat{G}_z|$ by $[1+O(\beta_z)]\hat{C}_{p_z}$.  With $f_1(z) =zN \leq 1 + O(\beta_z)$,
and with the definition of the Fourier transform in \eqref{e:FTdef},
we obtain
\begin{align}
    \sum_{x\in\Q^N} [1-(-1)^{k\cdot x}] |g_z(x) |
    & \leq  \sum_{x\in \Q^N} [1-(-1)^{k\cdot x}][zN D(x) + |\Pi_z(x)|]
    \nnb
    & \leq
    (1+O(\beta_z))[1-\hat{D}(k)]
    +  \sum_{x\in \Q^N} [1-(-1)^{k\cdot x}] |\Pi_z(x)|
    .
\label{e:fkl3}
\end{align}
In the last term the absolute values inside the sum prevent a direct application
of \eqref{e:diffPibd}, but the proof of \eqref{e:diffPibd} bounds the sum over $M$ absolutely,
so \eqref{e:diffPibd} also holds for the above sum and we see that
\begin{align}
\label{e:fkl4}
    \sum_{x\in\Q^N} [1-(-1)^{k\cdot x}] |g_z(x) |
    & \leq
    [1+O(\beta_z)][1- \hat D(k)].
\end{align}
Together, for all $k,\ell\in\Q^N$ these bounds give
\begin{equation}
    |\hat G_{z,k}(\ell)|
    \le
    [1+O(\beta_z)][1- \hat D(k)] \hat C_{p_z}(\ell) \hat C_{p_z}(k+\ell)
    =
    [1+O(\beta_z)] \bar C_{p_z}(k,\ell)
\end{equation}
which implies that $f_3(z) \leq 1+O(\beta_z)$.

\smallskip \noindent
This completes the proof that $f(z) \leq 1+O(\beta_z)$.
\end{proof}

\section{Consequences of lace expansion convergence}
\label{sec:pfF}

In this section, we first prove Theorems~\ref{thm:chiN-new} and \ref{thm:EL-new},
which follow from a well-known differential inequality together with the
bubble condition discussed in Remark~\ref{rk:bubble}.
We also complete the proofs of Theorems~\ref{thm:cnbd}
and \ref{thm:dilute-short}
by proving Propositions~\ref{prop:FzN}, \ref{prop:alphabeta} and  \ref{prop:F}
which we have seen in Section~\ref{sec:chi} imply Theorems~\ref{thm:cnbd}
and \ref{thm:dilute-short}.  The proofs of Propositions~\ref{prop:alphabeta} and  \ref{prop:F}
make use of fractional derivatives.
Once this has been accomplished only Theorem~\ref{thm:expansion} remains; its
proof is given in Section~\ref{sec:expansion}.

Note that $F_N(z)$ which was natural notation in Section~\ref{sec:chi}
is identical to $\hat F_z(0)$ which is natural in the
context of the lace
expansion where we also used the Fourier transform $\hat F_z(k)$
(recall \eqref{e:Fzhatdef}).
In this section, we favour the notation $F_N(z)$ since the Fourier transform
reappears only within the proof of Lemma~\ref{lem:Hfracder}.

Our concern is with small positive values of the parameter $\lambda$
used to define
$z_N=z_N(\lambda)$ by $\chi_N(z_N)=\lambda V^{1/2}$.  By definition of
$\beta_z$ in \eqref{e:betazdef}, we see that for $z \in [0,z_N]$ we have
\begin{equation}
    \beta_z \le N^{-1}+\lambda^2 .
\end{equation}
We take $\lambda \in (0,\lambda_0]$ where $\lambda_0$ will be chosen to be sufficiently small.

\subsection{Proof of Theorems~\ref{thm:chiN-new} and \ref{thm:EL-new}}

To prove Theorems~\ref{thm:chiN-new} and \ref{thm:EL-new}, we recall a differential inequality
from \cite{BFF84} which we state here as in \cite[Theorem~2.3]{Slad06}
(or \cite[Lemma~1.5.2]{MS93}), namely
\begin{equation}
\label{e:diffeq}
    \frac{1}{\bubble(z)} \chi_N(z)^2
    \le
     \partial_z[z\chi_N(z)]
    \le
    \chi_N(z)^2
\end{equation}
where
\begin{equation}
    \bubble(z) = \|G_z\|_2^2
\end{equation}
is the bubble diagram introduced in Remark~\ref{rk:bubble}.
The upper bound is a very elementary consequence of submultiplicativity, as we discuss
below \eqref{e:Gderivbds}.  Both inequalities in \eqref{e:diffeq} hold
on any finite or infinite transitive graph, but for the lower bound to be useful
it is necessary to have control of the bubble diagram.

\begin{proof}[Proof of Theorem~\ref{thm:EL-new}]
Let $z \in [0,z_N]$.
By \eqref{e:diffeq}, the definition \eqref{e:ELdef} of the expected length,
and the monotonicity of the bubble diagram,
\begin{equation}
    \frac{1}{\bubble(z_N)} \chi_N(z)
    \le
    \E_z^{(N)}L
    \le
    \chi_N(z)  .
\end{equation}
This gives the upper bound claimed in \eqref{e:ELasy-new}, and the lower bound
follows from the fact that $\|G_{z_N}\|_2^2 = 1 + \|H_{z_N}\|_2^2 = 1+ O(\beta_{z_N})$
(see Remark~\ref{rk:bubble}).
\end{proof}

\begin{proof}[Proof of Theorem~\ref{thm:chiN-new}]
The inequality \eqref{e:diffeq} can equivalently be written
in terms of the reciprocal $F_N$ of $\chi_N$ as
\begin{equation}
\label{e:diffeq1}
    \frac{1}{\bubble(z)} - F_N(z)  \le -z \partial_z F_N(z) \le 1 -F_N(z).
\end{equation}
Integration of the upper bound over the interval $[z,w]$ with $w > z >0$ leads to
\begin{equation}
    \log\left( \frac{1-F_N(w)}{1-F_N(z)}\right) \le \log\left( \frac{w}{z}\right),
\end{equation}
which rearranges to the general lower bound
\begin{equation}
\label{e:chiasy-general-lb}
    \chi_N(z)
    \ge
    \frac{1}{\chi_N(w)^{-1}z/w + 1-z/w}.
\end{equation}
The choice $w=z_N$, together with replacement of $z/w$ by $1$ for the first ratio
in the denominator to give a further lower bound,
proves the lower bound of \eqref{e:chiasy-new}.

For the lower bound of \eqref{e:diffeq1}, for $z\in [z',z_N]$ we observe that
\begin{equation}
    -z_N\partial_z F_N(z) \ge \frac{1}{\bubble(z_N)} - F_N(z')
\end{equation}
and integrate to obtain
\begin{equation}
    z_N(F_N(z')-F_N(z_N)) \ge \Big( \frac{1}{\bubble(z_N)} - F_N(z')\Big) (z_N-z').
\end{equation}
After replacement of $z'$ by $z$, this rearranges to
\begin{equation}
    (2z_N-z)F_N(z) \ge z_NF_N(z_N) + \frac{1}{\bubble(z_N)} (z_N-z)
\end{equation}
which is equivalent to the upper bound on $\chi_N$ of
\eqref{e:chiasy-new} since $\bubble(z_N) = 1+ O(\beta_{z_N})$.
\end{proof}

\subsection{First derivative of  $\Pi$}
\label{sec:dzPi}

Next, we prove a bound
on the $z$-derivative $\partial_z\hat\Pi_z(k)$.
For the proof of Proposition~\ref{prop:F} we will also
consider fractional derivatives in Section~\ref{sec:pfpropF}.

We begin with the observation that for any $z \ge 0$ and $j\in\N$,
\begin{equation}
\label{e:Gderivbds}
	z^j \partial_z^j H_z (x)  \leq j! (H_{z}^{*j} * G_z) (x) .
\end{equation}
A proof can be found in \cite[Lemma~6.2.8]{MS93} and we illustrate the idea for
$j=1$ as follows (we will only use $j=1,2$).  For $j=1$,
\begin{equation}
    z \partial_z H_z (x)
    =
    \sum_{n=1}^\infty nc_n^{(N)}(x) z^n = \sum_{n=1}^\infty \sum_{i=1}^n c_n^{(N)}(x) z^n
    \le \sum_{n=1}^\infty \sum_{i=1}^n (c_i^{(N)}*c_{n-i}^{(N)})(x) z^iz^{n-i},
\end{equation}
since $c_n^{(N)}(x) \le (c_i^{(N)}*c_{n-i}^{(N)})(x)$ follows by relaxing the self-avoidance constraint
between the first $i$ and last $n-i$ steps.
After interchange of sums, the right-hand side rearranges
to $(H_z*G_z)(x)$.  This is in fact the essential step in the proof of the upper bound
of \eqref{e:diffeq}.

\begin{lemma}
\label{lem:Pizder}
There is a $\lambda_0>0$ such that for any $\lambda\in (0,\lambda_0]$,
and for all $z \in \C$ with $|z| \le z_N$ and for
all $k \in \Q^N$,
\begin{align}
\label{e:dPi}
    |\partial_z\hat\Pi_z(k)| \prec N\beta_{|z|}.
\end{align}
\end{lemma}

\begin{proof}
As in \eqref{e:PizC}, it suffices to prove that there is a constant $c$ such that, for all real $z \in [0,z_N]$,
\begin{align}
\label{e:dPiM}
    \sum_{M=1}^\infty \|\partial_z \Pi_z^{(M)}\|_1 & \le cN\beta_z
    .
\end{align}
Since $\delta_z^1 = z \partial_z$, it follows  from Proposition~\ref{prop:Pidots}
with $\epsilon=1$ that for
$z \ge 0$ and $M=1$,
\begin{equation}
\label{e:Pi1bd-bis}
    \|  \partial_z  \Pi^{(1)}_z \|_1
    \le N\|\partial_z (zH_z)\|_\infty
    ,
\end{equation}
and for $z \ge 0$ and $M \ge 2$,
\begin{align}
\label{e:Pizder-bis}
    \|  \partial_z   \Pi^{(M)}_z \|_1
    &\le
    (2M-1)
    \|\partial_z H_z\|_\infty
    \|H_z * G_z \|_\infty^{M-1}
    \nnb & \le
    (2M-1)
    \|\partial_z  H_z\|_\infty
    (2c_2 \beta_z)^{M-1}
    .
\end{align}
where we used the bound $\|H_z * G_z \|_\infty \le  2c_2 \beta_z$
from \eqref{e:HstarG} (with $K=2$) for the last inequality.

By \eqref{e:Gderivbds} and Lemma~\ref{lem:TDa},
\begin{align}
\label{e:zdH}
    \partial_z (zH_z(x)) & = H_z(x) + z\partial_z H_z(x)
    \le
    H_z(x) + (H_z*G_z)(x) \prec \beta_z,
\end{align}
and with \eqref{e:Pizder-bis}
this gives an $O(N\beta_z)$ bound for the $M=1$ term in \eqref{e:dPiM}.

For $M \ge 2$, a slight manoeuvre  is needed to deal with small $z$ due to the fact
that we desire a bound on $\partial_zH_z$ whereas \eqref{e:Gderivbds} provides a
bound on $z\partial_zH_z$.
Suppose that $z \in [0,\frac{1}{2N}]$.  In this case,
as in \eqref{e:HDG} we use $H_z(x) \le zN(D*C_z)(x)$.  This inequality in fact
holds term-by-term as power series, so it is preserved upon differentiation and
\begin{align}
    \partial_z H_z(x)
    & \le N(D*C_z)(x) + zN\partial_z (D*C_z)(x).
\end{align}
Since $C_z(x)= \delta_{0,x}+ zN(D*C_z)(x)$,
and since $\hat C_z(k) \le \hat C_z(0)\le 2$ for $z \le \frac{1}{2N}$,
the first term on the above right-hand side
obeys
\begin{align}
    N(D*C_z)(x) & \le N \Big( D(x) + zN(D*D*C_z)(x) \Big)
    \nnb &
    \le 1 + \frac 12 N \|\hat D(k)^2 \hat C_z(k)\|_{\hat 1}
    \prec  1 +  N \|\hat D(k)^2 \|_{\hat 1} \prec N\beta_z,
\end{align}
since the norm on the right-hand side is of order $N^{-1}$ by Lemma~\ref{lem:Dsum}.
Similarly, with \eqref{e:Cphatdef} we see that
\begin{align}
    zN\partial_z (D*C_z)(x)
    =
    zN\frac{1}{V} \sum_{k \in \Q^N} \frac{N\hat D(k)^2(-1)^{k\cdot x}}{[1-zN\hat D(k)]^2}
    \prec
    \frac{N}{V} + N\cdot N^{-1} \prec N \beta_z,
\label{e:Hzder-small}
\end{align}
where we separated the $k=0$ term and used Lemma~\ref{lem:Dsum} to bound the sum over
nonzero $k$.
On the other hand,
for $z \in [\frac{1}{2N},z_N]$ it follows from \eqref{e:zdH}
that
\begin{equation}
\label{e:Hzder}
    \partial_z H_z(x) \le 2Nz \partial_z H_z(x) \prec N \beta_z.
\end{equation}
Thus, using the above bound $N\beta_z$ for the $M=1$ term, together
with the above considerations to bound $\|\partial_zH_z\|_\infty$  by $N\beta_z$
when $M \ge 2$, altogether we find from \eqref{e:Pizder-bis} that
\begin{align}
\label{e:PiMzder-pf}
    \sum_{M=1}^\infty \|\partial_z \Pi_z^{(M)}\|_1 & \prec
    N\beta_z
    +
    N\beta_z  \sum_{M=2}^\infty M\, (2c_2\beta_z)^{M-1}
    \prec N \beta_z
    ,
\end{align}
where we used small $\lambda_0$ to bound the last sum.
\end{proof}

We can now easily prove \eqref{e:muratio}, as follows.
As a first and elementary observation,
since the generating function for self-avoiding walks is smaller than the
generating function for all walks, we have $\chi_N(z) \le \hat C_z(0) = (1-zN)^{-1}$
for $z < \frac 1N$,
and hence $z_N$ is larger than the value $p_{z_N}=N^{-1}(1-\lambda^{-1}V^{-1/2})$ for which $(1-p_{z_N}N)^{-1}=\lambda V^{1/2}$ (recall Figure~\ref{fig:pofz}).
This shows that $z_N \ge \frac 34 N^{-1}$
(we take $V$ large depending on $\lambda$) and therefore
\begin{equation}
\label{e:zNlam}
    \frac 34 N^{-1}  \le z_N \le 2N^{-1}
\end{equation}
since we have seen in Lemma~\ref{lem:43} that $z_N N\le 2$.

Since $F_N(z_N(\lambda)) = \lambda^{-1} V^{-1/2}$, the chain rule
and the formula for $F_N(z)$ in \eqref{e:chiPi} give
\begin{equation}
    z_N'(\lambda) = \frac{1}{-F_N'(z_N(\lambda))} V^{-1/2} \frac{1}{\lambda^2}
    = \frac{1}{N+\partial_z \hat\Pi_{z}(0)|_{z_N(\lambda)}} V^{-1/2} \frac{1}{\lambda^2}.
\end{equation}
The first fraction on the right-hand side is at most $2N^{-1}$ (for small $\lambda_0$)
by Lemma~\ref{lem:Pizder},
so for $\lambda' \le \lambda_1< \lambda_2 \le \lambda_0$,
\begin{equation}
    z_N(\lambda_2) - z_N(\lambda_1) = \int_{\lambda_1}^{\lambda_2} z_N'(\lambda) \D \lambda
    \prec
    \frac{V^{-1/2}}{N} \int_{\lambda_1}^{\lambda_2} \frac{1}{\lambda^2} \D \lambda
    =
    \frac{V^{-1/2}}{N} \frac{\lambda_2-\lambda_1}{\lambda_1\lambda_2}
\end{equation}
and hence, as claimed in \eqref{e:muratio},
\begin{equation}
\label{e:zratio}
    \frac{z_N(\lambda_2)}{z_N(\lambda_1)} - 1
    \prec
    \frac{V^{-1/2}}{Nz_N(\lambda_1)} \frac{\lambda_2-\lambda_1}{\lambda_1\lambda_2}
    \prec
     \frac{\lambda_2-\lambda_1}{\lambda_1\lambda_2}
     V^{-1/2}  .
\end{equation}
In the last step, we used the lower bound of \eqref{e:zNlam}.
As usual, the constant in \eqref{e:zratio} deteriorates as $\lambda_1,\lambda_2$
decrease, so depends on $\lambda'$.

\subsection{Proofs of Propositions~\ref{prop:FzN} and \ref{prop:alphabeta} }
\label{sec:manypfs}

We are now in a position to prove
Propositions~\ref{prop:FzN} and \ref{prop:alphabeta}.
The following proposition concerns the susceptibility
in the
complex plane.

\begin{prop}
\label{prop:chiasy-bis}
There is a $\lambda_0>0$ such that for any $\lambda\in (0, \lambda_0]$,
and for any $w_N \in \C$ with $|w_N|\le z_N$ and  $\lim_{N\to\infty}V^{1/2}|1-w_N/z_N| =\infty$,
\begin{equation}
\label{e:chiasy-bis}
    |\chi_N(w_N)|
    \asymp \frac{1}{|1-w_N/z_N|}  .
\end{equation}
In addition, for $z \in \C$ with $|z| \le z_N$, if $N$ is large then
\begin{align}
    |\chi_N(z)| &\le
    \frac{2}{ |1-z/z_N|}
    .
\label{e:Flower1}
\end{align}
\end{prop}

\begin{proof}
Let $|z|\le z_N$.
By the Fundamental Theorem of Calculus and \eqref{e:chiPi},
\begin{equation}
\label{e:FTCz}
    F_{N}(z) =
    F_{N}(z_N)
    +
    z_N N (1-z/z_N) + \int_{z}^{z_N} \partial_w \hat\Pi_w(0) \D w,
\end{equation}
with the integral along the line segment joining $z$ to $z_N$.
By Lemma~\ref{lem:Pizder}, the integral on the right-hand side is
$O[N\beta_{z_N}|z_N-z|] = O[(N^{-1}+\lambda^2)|1-z/z_N|]$, since $z_N \le 2N^{-1}$.
This gives
\begin{equation}
\label{e:Fnzlb1}
    F_{N}(z) =
    \lambda^{-1}V^{-1/2}
    +
    z_N N(1-z/z_N)  +O[(N^{-1}+\lambda^2)|1-z/z_N|]  .
\end{equation}
Since $z_NN \ge \frac 34$,
the last term is comparable to the middle term but with much smaller prefactor.
When we set $z=w_N$ with the assumption that
$\lim_{N\to\infty}V^{1/2}|1-w_N/z_N| =\infty$,
the term $\lambda^{-1}V^{-1/2}$ becomes negligible and \eqref{e:chiasy-bis} follows.

Also, it follows from \eqref{e:Fnzlb1} that, with $\epsilon = (z_NN\lambda V^{1/2})^{-1}$,
\begin{align}
\label{e:Fnzlb2}
    |F_{N}(z)| &\ge
    z_NN \big|\epsilon
    +
    1-z/z_N  \big| - O[(N^{-1}+\lambda^2)|1-z/z_N|]
    \nnb & \ge
     z_NN | 1-z/z_N  | - O[(N^{-1}+\lambda^2)|1-z/z_N|]
    ,
\end{align}
with the first inequality a consequence of the triangle inequality and the
second due to the
geometric fact that any point $z/z_N$ in the  unit disk is closer to $1$ than
it is to $1+\epsilon$.
Since $z_NN \ge \frac 34$, this gives \eqref{e:Flower1}
and completes the proof.
\end{proof}

The proofs of Propositions~\ref{prop:FzN} and \ref{prop:alphabeta}
now follow easily.

\begin{proof}[Proof of Proposition~\ref{prop:FzN}]
This is an immediate consequence of \eqref{e:zNlam} (for the bound $Nz_N \le 2$),
of \eqref{e:Flower1} (for the lower bound on $|F_N(z)|$ in the disk $|z|\le z_N$),
and of the fact that $F_N'(z) = -N - \partial_z \hat\Pi_z(0)$ together with
Lemma~\ref{lem:Pizder} which implies that $|\partial_z \hat\Pi_z(0)| \prec
N\beta_{z_N} \prec 1+ N\lambda^2$ uniformly in $|z|\le z_N$ (for the bound
$|F_N'(z)| \le 2 N$).
\end{proof}

\begin{proof}[Proof of Proposition~\ref{prop:alphabeta}]
Let $p \in (0,\frac 12)$ and recall that $\zeta_p=z_N(1-V^{-p})$.
Proposition~\ref{prop:alphabeta} asserts that
\begin{equation}
\label{e:zetapasy}
    \zeta_p = N^{-1}[1+O(N^{-1})], \qquad
    F_N(\zeta_p) \asymp V^{-p}, \qquad   F_N'(\zeta_p) = -N +O(1).
\end{equation}
For the second of these three statements, we first observe that by
definition $V^{1/2}(1-\zeta_p/z_N)=V^{1/2-p} \to \infty$.  It then follows from
\eqref{e:chiasy-bis} that
\begin{equation}
\label{e:chip4}
    \chi_N(\zeta_p) \asymp   V^p,
\end{equation}
which is equivalent to the desired relation $F_N(\zeta_p) \asymp V^{-p}$.
By the definition of $\beta_z$ in \eqref{e:betazdef}, this also implies that
\begin{equation}
\label{e:betap}
    \beta_{\zeta_p} \prec V^{2p-1}+N^{-1} \prec N^{-1}.
\end{equation}
By
\eqref{e:chiPi}
\begin{equation}
    \chi_N(\zeta_p)^{-1} =  1 - N\zeta_p -\hat\Pi_{\zeta_p}(0),
\end{equation}
so by Lemma~\ref{lem:Pibds}, \eqref{e:chip4} and \eqref{e:betap},
\begin{equation}
    \zeta_p =  N^{-1} [1  -\hat\Pi_{\zeta_p}(0) -\chi_N(\zeta_p)^{-1}]
    =  N^{-1} [1+  O(  N^{-1})],
\end{equation}
which proves the first statement of \eqref{e:zetapasy}.
Finally, since
\begin{equation}
    F_N'(\zeta_p) = - N - \partial_z \hat\Pi_z(0)\big|_{z=\zeta_p},
\end{equation}
the third statement of \eqref{e:zetapasy} follows from Lemma~\ref{lem:Pizder} and \eqref{e:betap}.
\end{proof}

\subsection{Fractional derivatives}
\label{sec:fracder}

The proof of Proposition~\ref{prop:F}
uses the fractional derivative methods introduced in
\cite{HS92a} (see also \cite[Section~6.3]{MS93}), and we begin with a summary
of what is needed from that theory.

The fractional derivative $\delta_z^\epsilon f(z)=\sum_{n=1}^\infty n^\epsilon a_n z^n$
defined in \eqref{e:depdef}
converges absolutely at least strictly
within the circle of convergence of $f(z)$.  For
$\rho>0$ and for $j =1,2$ we define
\begin{equation}
\label{e:fnorm}
	A_j^{(\epsilon)}(\rho) = \sum_{n=j}^\infty n^\epsilon |a_n| \rho^n.
\end{equation}
The following lemma
provides an error estimate analogous to the error
estimate in Taylor's theorem.  It is a restatement of \cite[Lemma~6.3.2]{MS93},
apart from the replacement of $A_1^{(\epsilon)}(\rho)$ in \cite[Lemma~6.3.2]{MS93}
by $A_2^{(\epsilon)}(\rho)$ in \eqref{e:Taylepbound}.
The replacement is justified since the
terms $a_0+a_1z$ in $f(z)$ cancel on the left-hand side of
\eqref{e:Taylepbound} and hence we can assume $a_0=a_1=0$ there.

\begin{lemma}
\label{lem:Taylepsilon}
Let $\epsilon \in (0,1)$, $f (z) = \sum_{n=0}^\infty a_n z^n$, and
$\rho >0$.  If
$A_1^{(\epsilon)}(\rho) < \infty$
(so in particular $f(z)$ converges for $|z|\leq \rho$),
then for any $z\in \C$ with $|z| \leq \rho$,
\begin{equation}
\label{e:Taylep0}
	|f(z) - f(\rho)  | \leq
	2^{1-\epsilon} A_1^{(\epsilon)}(\rho) |1-z/\rho|^{\epsilon}.
\end{equation}
If
$A_{2}^{(1+\epsilon)}(\rho) < \infty$
(so in particular $f'(z) = \sum_{n=1}^\infty na_nz^{n-1}$ converges for
$|z|\leq \rho$),
then for any $z\in \C$ with  $|z| \leq \rho$,
\begin{equation}
\label{e:Taylepbound}
	| f(z) - f(\rho) - f'(\rho) (z-\rho) | \leq
	\frac{2^{1-\epsilon}}{1+\epsilon}  A_{2}^{(1+\epsilon)}(\rho)
	|1-z/\rho|^{1+\epsilon}.
\end{equation}
\end{lemma}

For $\epsilon \in (0,1)$, the following
elementary lemma provides an integral formula for $\delta_z^\epsilon f$
in terms of $f'$.  The statement and proof of the lemma involve the Gamma function
$\Gamma(x)=\int_0^\infty t^{x-1}e^{-t}\D t$.

\begin{lemma}
\label{lem:frderiv}
Let $f(z) = \sum_{n=0}^\infty a_n z^n$ have radius of convergence
at least $\rho$.  Let $\epsilon \in (0,1)$ and $\gamma_\epsilon = 1/\Gamma(1-\epsilon)$.
For any  $z\in [0,\rho)$ (and also for $z=\rho$ if $a_n \ge 0$ for all $n$),
\begin{equation}
\label{e:frderividentity}
	\delta_z^\epsilon f(z) =
    \gamma_\epsilon
	z
	\int_0^\infty f'(ze^{-t})
	e^{-t} t^{-\epsilon} \D t.
\end{equation}
\end{lemma}

\begin{proof}
It is proved in \cite[Lemma~6.3.1]{MS93} that
\begin{equation}
\label{e:frderividentityMS}
	\delta_z^\epsilon f(z) =
	C_{\epsilon}\, z
	\int_0^\infty f'(ze^{-\lambda^{1/(1-\epsilon)}})
	e^{-\lambda^{1/(1-\epsilon)}} \D \lambda
\end{equation}
with $C_\epsilon = [(1-\epsilon)\Gamma(1-\epsilon)]^{-1}$.
We make the change of variables $t=\lambda^{1/(1-\epsilon)}$ to obtain
\eqref{e:frderividentity}.
\end{proof}

\subsection{Fractional derivatives of $\Pi$: proof of Proposition~\ref{prop:F}}
\label{sec:pfpropF}

We now prove Proposition~\ref{prop:F}.
The proof is based on an estimate for the $(1+a)^{\rm th}$
derivative of $\hat\Pi_z$ at $z=\zeta_p$.  This in turn requires a fractional-derivative bound on
$H_z$ and we begin with the crucial lemma that provides this bound.
The proof of Lemma~\ref{lem:Hfracder} requires delicate attention to the zero mode.
Constants in estimates are permitted to depend on $p$ and $a$ as well as on $\lambda$.

\begin{lemma}
\label{lem:Hfracder}
There is a $\lambda_0>0$ such that for any $\lambda\in (0, \lambda_0]$,
and for any $p \in (0,\frac 12)$ and $a\in (0,1)$,
\begin{align}
\label{e:Hfracder1}
    N\|\delta_z^{1+a}|_{z=\zeta_p}(zH_z)\|_\infty
    & \prec N^{-1} +   V^{(2+a)p-1},
    \\
\label{e:Hfracder2}
    \|\delta_z^{1+a}|_{z=\zeta_p} H_z\|_\infty & \prec  N^{-1} +  V^{(2+a)p-1}.
\end{align}
\end{lemma}

\begin{proof}
The values of $p \in (0,\frac 12)$ and $a\in (0,1)$ are fixed throughout the proof.
To simplify the notation, we will write simply $H_z$ in place of $H_z(x)$;
all upper bounds are uniform in $x \in \Q^N$.  Also, all derivatives $\partial_z$ and
$\delta_z$ are with respect to $z$.  Since we need to evaluate these derivatives
at both $\zeta_p$ and at $\zeta_pe^{-t}$, to lighten the notation we write
$\partial_{\zeta_p}$ and $\delta_{\zeta_p}$ for derivatives evaluated at
$\zeta_p$, and we introduce $\eta_t = \zeta_pe^{-t}$ and similarly write
$\partial_{\eta_t}$ and $\delta_{\eta_t}$.

For \eqref{e:Hfracder1} we use the general fact that
$\delta^{1+a}_zf(z)=\delta_z^a [z\partial_zf(z)]$ and
then \eqref{e:frderividentity} to see that
\begin{align}
\label{e:Hfracint}
    \delta_{\zeta_p}^{1+a} (zH_z)
    =
    \delta_{\zeta_p}^{a}[z\partial_z(zH_z)]
    & =
    \gamma_a \zeta_p \int_0^\infty \partial_{\eta_t} [z\partial_z(zH_z)]
    e^{-t} t^{-a} \D t.
\end{align}
The derivative on the right-hand side satisfies
\begin{align}
    \partial_z [z\partial_z(zH_z)]
    & =
    H_z + 3z \partial_z H_z + z^2 \partial_z^2 H_z
    \nnb & =
    c_1^{(N)}(x)z + 3c_1^{(N)}(x)z + z^2\sum_{n=2}^\infty
    (1 + 3n  +n(n-1))c_n^{(N)}(x)z^{n-2},
\end{align}
and hence since $c_1^{(N)}(x) \le 1$, and since $\eta_t \le \zeta_p \sim N^{-1}$ by \eqref{e:zetapasy},
\begin{align}
    \partial_{\eta_t} [z\partial_z(zH_z)]
    & \prec
    N^{-1}  + N^{-2} \partial_{\eta_t} ^2 H_z .
\end{align}
Therefore, by \eqref{e:Hfracint} and again using $\zeta_p\prec N^{-1}$,
we have the preliminary estimate
\begin{align}
    N\delta_{\zeta_p}^{1+a} (zH_z)
    & \prec
    N^{-1} +
    N^{-2} \int_0^\infty
    \partial_{\eta_t}^2 H_z \;
    e^{-t} t^{-a} \D t.
\end{align}

For \eqref{e:Hfracder2}, and for $z \le \zeta_p \prec N^{-1}$, as above
and with \eqref{e:Hzder} we have
\begin{align}
    \partial_z(z\partial_z H_z) = \partial_z H_z +z \partial_z^2 H_z
    \prec 1 + z \partial_z^2 H_z
    \prec
    1 + N^{-1} \partial_z^2 H_z,
\end{align}
so again (with $\zeta_p \prec N^{-1}$) we find that
\begin{align}
    \delta_{\zeta_p}^{1+a} H_z|_{z=\zeta_p} & =
    \delta_{\zeta_p}^a (z \partial_z H_z )  =
    \gamma_a \zeta_p \int_0^\infty \partial_{\eta_t} (z\partial_z H_z )  e^{-t}t^{-a} \D t
    \nnb &
    \prec
    N^{-1} +
    N^{-2} \int_0^\infty
    \partial_{\eta_t}^2 H_z \;
    e^{-t} t^{-a} \D t
    .
\label{e:dH2}
\end{align}
To complete the proof of \eqref{e:Hfracder1}--\eqref{e:Hfracder2},
it therefore suffices to prove that
\begin{align}
    \int_0^\infty  \partial_{\eta_t}^2  H_z   e^{-t} t^{-a} \D t
    & \prec N
    +
    N^2 V^{(2+a)p-1}.
\label{e:Hjder}
\end{align}
We consider small $t$ and
large $t$ separately.  The delicate part is small $t$.

For $z \le \frac{1}{2N}$, we first observe that since the number of $n$-step self-avoiding
walks is bounded above by the number of $n$-step simple random walks, we
have  $\partial_z^2 H_z(x)  \le  \partial_z^2 C_z(x)$.  We rewrite $C_z(x)$
using the inverse Fourier transform \eqref{e:IFT} in conjunction with the formula
for $\hat C_z(k)$ in \eqref{e:Cphatdef}, and thereby obtain
\begin{align}
    \partial_z^2 H_z(x) & \le  \partial_z^2 C_z(x)
     =
    \partial_z^2 \frac{1}{V} \sum_{k \in \Q^N} \frac{(-1)^{k\cdot x}}{1-zN\hat {D}(k)}
    \nnb &
    = \frac{2N^2}{V} \sum_{k \in \Q^N} \frac{\hat{D}(k)^2 (-1)^{k\cdot x}}{(1-zN\hat {D}(k))^3}
    \le \frac{2N^2}{V} \sum_{k \in \Q^N}  \frac{\hat{D}(k)^2}{(1-\frac 12)^3}
    \prec N,
\end{align}
where we used $zN|\hat D(k)| \le zN \le \frac 12$ in the penultimate step
and used \eqref{e:tranprob} (with $x=0$) in the last step.
Thus in \eqref{e:Hjder} the integral over $t$ such that $\eta_t \le \frac{1}{2N}$
is bounded by $N$ as required.  Such $t$ include those for which $e^{-t} \le \frac 14$
since $N\zeta_p \le 2$.
It therefore suffices to prove that
\begin{align}
     \int_0^{\log 4} \partial_{\eta_t}^2 H_z \;  t^{-a} \D t
      & \prec
      N+  N^2 V^{(2+a)p-1}
     ,
\label{e:fracint4}
\end{align}
where
we have neglected a now unimportant exponential factor in the integrand using $e^{-t} \le 1$.

The derivative in the integrand of \eqref{e:fracint4} can be bounded using
\eqref{e:Gderivbds}, with the result that
\begin{align}
     \int_0^{\log 4} \partial_{\eta_t}^2 H_z(x) \,  t^{-a} \D t
      & \prec
      \int_0^{\log 4}
     (\zeta_pe^{-t})^{-2}(H_{\eta_t}* H_{\eta_t}*G_{\eta_t})(x) \, t^{-a} \D t
     \nnb & \prec
     N^{2} \int_0^{\log 4}
     (H_{\eta_t}* H_{\eta_t}*G_{\eta_t})(x) \, t^{-a} \D t
     .
\label{e:fracint}
\end{align}
Now we use $\|f\|_\infty \le \|\hat f\|_{\hat 1}$ to see that
\begin{align}
\label{e:HHGsum}
    (H_{\eta_t}* H_{\eta_t}*G_{\eta_t})(x)
    & \le
    \frac 1V \sum_{k \in \Q^N} \hat H_{\eta_t}(k)^2  \hat G_{\eta_t}(k)
    .
\end{align}
Recall from \eqref{e:Fzhatdef} that we write $\hat F_z(k)=1/\hat G_z(k)$.
As in \eqref{e:gdef}, we define $\hat g_z(k)$ by
\begin{align}
    \hat H_z(k) = \hat G_z(k) - 1 =
    \frac{zN\hat D(k)+\hat\Pi_z(k)}{\hat F_z(k)}
    =
    \frac{\hat g_z(k)}{\hat F_z(k)} = \hat g_z(k) \hat G_z(k).
\end{align}
For $z\le \zeta_p$, since $zN \prec 1$ and $|\hat\Pi_z(k)| \prec \beta_{\zeta_p}
\prec N^{-1}$ (recall \eqref{e:betap}),
\begin{equation}
    \hat g_z(k)^2 \prec \hat D(k)^2 + \hat\Pi_z(k)^2 \prec \hat D(k)^2 +N^{-2}.
\end{equation}
We know from the bootstrap proof in Section~\ref{sec:bootstrap} that $f_2(z) \le 2$
for all $z \le z_N$, so
\begin{equation}
    \hat G_z(k) \le 2 \hat C_{p_z}(k) = \frac{2}{1- p_zN\hat D(k)}.
\end{equation}
The sum over nonzero $k$ in \eqref{e:HHGsum} is therefore at most
\begin{align}
    \frac 1V \sum_{k \neq 0} \hat g_{\eta_t}(k)^2  \hat G_{\eta_t}(k)^3
    & \prec
    \frac 1V \sum_{k \neq 0} \frac{\hat D(k)^2}{[1-p_{\eta_t}N \hat D(k)]^3} +
    \frac 1V \sum_{k \neq 0} \frac{N^{-2}}{[1-p_{\eta_t} N\hat D(k)]^3}
    \nnb &
    \prec N^{-1}+N^{-2} \prec N^{-1},
\end{align}
where we used Lemma~\ref{lem:Dsum} for the second line.
When we insert this into the right-hand side of \eqref{e:fracint} the result is $N^{+1}$
which is consistent with \eqref{e:fracint4}.

It remains to consider the zero mode.
The $k=0$ term in the sum in \eqref{e:HHGsum} is
\begin{equation}
    \frac 1V  \hat g_{\eta_t}(0)^2\chi_N(\eta_t)^3
    \prec
    \frac 1V \left( 1+N^{-2} \right) \chi_N(\eta_t)^3,
\end{equation}
so it now suffices to prove that
\begin{align}
    \int_0^{\log 4}
    \chi_N(\eta_t)^3 t^{-a} \D t \prec V^{(2+a)p }.
\end{align}
By \eqref{e:Flower1} and the fact that $\eta_t = z_N(1-V^{-p})e^{-t}$ by definition,
for $t \le \log 4$ we have
\begin{align}
    \chi_N(\eta_t) & \le \frac{2}{1-\eta_t/z_N}
    =
    \frac{2}{V^{-p}e^{-t} + 1-e^{-t}}
    \prec
    \frac{1}{V^{-p} + t}.
\end{align}
It follows that (with the change of variables $t=s V^{-p}$)
\begin{align}
    \int_0^{\log 4} \chi_N(\eta_t)^3
	 t^{-a} \D t
    &  \prec
    \int_0^{\log 4} \frac{1}{(V^{-p} + t)^3}
	t^{-a} \D t
    \prec
    \frac{1}{(V^{-p})^{2+a}}
    \int_0^{\infty} \frac{1}{( 1+ s)^3} s^{-a} \D s
    \prec V^{(2+a)p}
    .
\end{align}
This completes the proof.
\end{proof}

\begin{cor}
\label{cor:Pizder}
There is a $\lambda_0>0$ such that for any $\lambda\in (0, \lambda_0]$,
and for any $p \in (0,\frac 12)$ and $a\in (0,1)$,
\begin{align}
\label{e:deltaPi}
    \sum_{M=1}^\infty \|\delta_{\zeta_p}^{1+a}   \Pi_z^{(M)} \|_1
    & \prec
    N^{-1} +V^{(2+a)p-1} .
\end{align}
\end{cor}

\begin{proof}
It was shown in Proposition~\ref{prop:Pidots} that for
$z \ge 0$  and $\epsilon \ge 1$ we have
\begin{equation}
\label{e:Pi1bd-bis2}
    \|  \delta_z^\epsilon   \Pi^{(1)}_z \|_1
    \le N\|\delta_z^\epsilon (zH_z)\|_\infty,
\end{equation}
and for $M \ge 2$,
\begin{align}
\label{e:Pizder-bis2}
    \|  \delta_z^\epsilon   \Pi^{(M)}_z \|_1
    &\le
    (2M-1)^{\epsilon}
    \|\delta_z^{\epsilon} H_z\|_\infty
    \|H_z * G_z \|_\infty^{M-1}
    .
\end{align}
Also, by \eqref{e:HstarG},
\begin{equation}
    \|H_z * G_z \|_\infty \le  2c_K \beta_z .
\end{equation}
Then the desired result follows by using Lemma~\ref{lem:Hfracder}
to estimate the fractional derivatives of $H_z$, with small $\lambda_0$
used to control the sum over $M$.
\end{proof}

We are now ready to prove Proposition~\ref{prop:F} which we restate here for
convenience as Proposition~\ref{prop:F-bis}.
Recall that by definition
\begin{equation}
    F_N(z) = \Phi_N(z) + R_N(z)
\end{equation}
where $\Phi_N(z)$ is the linear function
\begin{equation}
    \Phi_N(z) = F_N(\zeta_p) + F_N'(\zeta_p)(z-\zeta_p).
\end{equation}

\begin{prop}
\label{prop:F-bis}
There is a $\lambda_0>0$ such that for any $\lambda\in (0, \lambda_0]$,
and for any $p \in (0,\frac 12)$, any $a\in (0,1)$,
and any $z\in\C$ with $|z| \le \zeta_p$,
\begin{align}
\label{e:frac-der-OK-bis}
    |R_N(z)| &\prec  (N^{-1}+ V^{(2+a)p-1}) |1-z/\zeta_p|^{1+a}   ,
\\
\label{e:varphiderbd-bis}
    |R_N'(z)|  &\prec (1+NV^{(2+a)p-1})  |1-z/\zeta_p|^a
    .
\end{align}
\end{prop}

\begin{proof}
For \eqref{e:frac-der-OK-bis} we
apply the fractional Taylor error estimate \eqref{e:Taylepbound} to
\begin{equation}
    f(z)=F_N(z)= 1-zN - \hat\Pi_{z}(0)
\end{equation}
with $\epsilon=a$ and $\rho=\zeta_p$.
The linear part $1-zN$ on the right-hand side does not contribute to
$A_{2}^{(1+a)}(\zeta_p)$ defined by \eqref{e:fnorm}, which therefore only involves
$\hat\Pi_{\zeta_p}(0)$ and according to Corollary~\ref{cor:Pizder} is
bounded by $N^{-1}+V^{(2+a)p-1}$.  Thus the claimed bound on $R_N(z)$ follows from
 \eqref{e:Taylepbound}.

For \eqref{e:varphiderbd-bis}, we first use the definition of
$R_N(z)$ to see that
\begin{equation}
    R_N'(z) = F_N'(z) - F_N'(\zeta_p)
    = \partial_{\zeta_p} \hat\Pi_z(0) - \partial_z \hat\Pi_z(0) .
\end{equation}
Since
\begin{equation}
    \partial_z \hat\Pi_z(0) = \sum_{m=1}^\infty (m+1) \sum_{x\in\Q^N} \pi_{m+1}(x) z^m,
\end{equation}
by \eqref{e:Taylep0} it suffices to have the upper bound
\begin{align}
    \sum_{m=1}^\infty m^a (m+1) \|\pi_{m+1}\|_1 \zeta_p^m
    & \prec
   N \sum_{m=1}^\infty (m+1)^{1+a} \|\pi_{m+1}\|_1 \zeta_p^{m+1}
   \prec
   1+NV^{(2+a)p-1}
    ,
\end{align}
which follows from Corollary~\ref{cor:Pizder}.  This completes the proof.
\end{proof}

\section{$1/N$ expansion: proof of Theorem~\ref{thm:expansion}}
\label{sec:expansion}

In this section we prove Theorem~\ref{thm:expansion}, which in particular states that
for $m \geq 1$ and $c>0$
(independent of $N$ but possibly depending on $m$), and for any $z$ such that
$\chi_N(z) \in [cN^{m}, \lambda_0 V^{1/2}]$, there are integers
$a_n$, which are universal constants that
are independent of the particular choice of $z$, such that
\begin{equation}
\label{e:aeqn-bis}
    z = \sum_{n=1}^m a_n N^{-n} + O(N^{-m-1})
    \quad
    \quad
    \mbox{as $N \to \infty$}.
\end{equation}
As usual, $\lambda_0$ is assumed to be sufficiently small.
The constant in the error term depends on $m,\lambda_0,c$, but does not
depend otherwise on $z$.
Theorem~\ref{thm:expansion} also includes a similar statement concerning an asymptotic
expansion for the amplitude $A_N$, and identifies the first five coefficients in each expansion.

\subsection{Independence of $z$}
\label{sec:z-indep}

We first establish the independence of $z$ for the coefficients $a_n$, assuming they exist.

\begin{prop}
\label{prop:p}
Let $c>0$, and let $\lambda_0>0$ be sufficiently small.
Let $m \ge 1$.
If $z$ obeys \eqref{e:aeqn-bis} for some
$z$ such that $\chi_N(z)\in
[cN^{m}, \lambda_0V^{1/2}]$, then \eqref{e:aeqn-bis} is valid for every
such $z$.
\end{prop}

\begin{proof}
Let $z'$ and $z_N$ be the solutions to $\chi_N(z') =cN^{m}$
and $\chi_N(z_N)=\lambda_0 V^{1/2}$.
For large $N$, $z' \leq z_N$.
It suffices to prove that
\begin{equation}
    z_N-z' \le O(N^{-m-1}).
\end{equation}
By \eqref{e:Flower1},
\begin{align}
    1-z'/z_N & \le \frac{2}{\chi_N(z')} = \frac{2}{cN^m},
\end{align}
and therefore, since $z_N \le 2 N^{-1}$ by \eqref{e:zNlam},
\begin{equation}
        z_N-z' \le z_N \frac{2}{cN^m} \le\frac{4}{cN^{m+1}}.
\end{equation}
This completes the proof.
\end{proof}

\subsection{Lace expansion estimates}
\label{sec:morePi}

To simplify the notation, for $k \ge 2$ and $M \ge 1$
we define
\begin{equation}
    \pi_k^{(M)} = \sum_{x\in\Q^N} \pi_k^{(M)}(x).
\end{equation}
Thus $\pi_k^{(M)}$ counts the number of $k$-step $M$-loop lace graphs depicted in
Figure~\ref{fig:lacetop}, with no restriction to end at a specific vertex $x$
and with the mutual avoidance of subwalks dictated by the
compatible edges in \eqref{e:PiMdef}.

The following proposition  prepares two estimates for the proof of
Theorem~\ref{thm:expansion}.
Its first inequality
will be used to see that  we can neglect large $M$  in the computation of the asymptotic expansions
of Theorem~\ref{thm:expansion} to a certain order,
and the second says that for small $M$
we can neglect large $k$.

\begin{prop}
\label{prop:Piforexpansion}
Let $\lambda_0>0$ be sufficiently small and let $p \in (0,\frac 12)$
and $\zeta_p=z_N(\lambda_0)(1-V^{-p})$.
For each $j \geq 1$ there is
a constant $C_j$, independent of $N$, such that
\begin{equation}
\label{e:d-N}
    \sum_{k=2}^\infty \sum_{M=j}^\infty k\pi_k^{(M)} \zeta_p^k \leq C_j N^{-j}.
\end{equation}
For each $j \geq 2$ and $M \geq 1$,
there is a constant $C_{M,j}$, independent of $N$, such that
\begin{equation}
\label{e:d-j}
    \sum_{k=j}^\infty k\pi_k^{(M)} \zeta_p^k \leq C_{M,j} N^{-\lceil j/2 \rceil}
    .
\end{equation}
\end{prop}

\begin{proof}
The inequality \eqref{e:d-N} follows from the proof of Lemma~\ref{lem:Pizder},
together with the facts that here we have an extra factor $\zeta_p \prec N^{-1}$
compared to $\partial_z\Pi_z^{(M)}$ and that $\beta_{\zeta_p} \prec N^{-1}$
by \eqref{e:betap}.
Indeed, the inequality \eqref{e:PiMzder-pf} is a term-by-term bound
which gives
\begin{align}
    \|\partial_z \Pi_z^{(M)}\|_1 & \prec
    N\beta_z M(2c_2\beta_z)^{M-1}
    ,
\end{align}
and hence (with $j$-dependent constants)
\begin{align}
    \sum_{k=2}^\infty \sum_{M=j}^\infty k\pi_k^{(M)} \zeta_p^k
    =
    \sum_{M=j}^\infty \zeta_p \|\partial_z \Pi_z^{(M)}\|_1 \prec
    \zeta_p \beta_{\zeta_p}^{j-1} \prec N^{-j}.
\end{align}
It remains to prove  \eqref{e:d-j}.
Constants are permitted to depend on $M$ and $j$ in the rest of the proof.

For \eqref{e:d-j}, we first obtain preliminary estimates using Lemma~\ref{lem:Dsum}.
The bound on
the random walk transition probability from \eqref{e:tranprob} asserts that for any integer $i \ge 0$ we have
\begin{equation}
\label{e:Dinf}
    \|D^{*i}\|_\infty \leq  c_i N^{-\lceil i/2 \rceil}.
\end{equation}
To avoid a notational clash, we temporarily write the value $p_z$
defined by \eqref{e:pzdef} as $q_z$.
Then, by using the bound $\hat G_{\zeta_p}(k) \le 2 \hat C_{q_{\zeta_p}}(k)$ obtained
from the bootstrap estimate $f_2(z) \le 2$ (recall \eqref{e:f3def}), we see that
\begin{align}
    \|\hat D^i \hat G_{\zeta_p}\|_{\hat 2}^2
    \le 4 \|\hat D^i \hat C_{q_{\zeta_p}}\|_{\hat 2}^2
    .
\end{align}
By extracting the zero mode, we
can conclude from  the fact that
$\hat C_{q_{\zeta_p}}(0)= \chi_N(\zeta_p)$, together with \eqref{e:chip4} and the second bound of
Lemma~\ref{lem:Dsum}  that, for any integer $i \ge 0$,
\begin{align}
\label{e:DGbd}
    \|\hat D^i \hat G_{\zeta_p}\|_{\hat 2}^2
    &\le
    \frac{4}{V} \chi_N(\zeta_p)^2
    +
    \frac{4}{V}\sum_{k \neq 0} \frac{\hat D(k)^{2i}}{[1-q_{\zeta_p}N\hat D(k)]^2}
    \prec \frac{1}{V^{1-2p}} + N^{-i} \prec N^{-i}.
\end{align}

Consider first the case $M=1$ of \eqref{e:d-j}.
Recall that $\Pi_z^{(1)}(x)$ is nonzero only for $x=0$,
in which case it is the generating function for self-avoiding returns to the origin.
We are interested in returns that take  $k \ge j$ steps, and the factor $k$ in
\eqref{e:d-j} counts the number of ways to label the distinct $k$ vertices of the
walk forming the return.  By relaxing the avoidance constraint for the first $j$ steps,
and also the avoidance between the remaining two parts of the walk separated by the
labelled vertex (in the worst case that it is not among the first $j$ steps) then we see that
\begin{equation}
    \sum_{k=j}^\infty k\pi_k^{(M)} \zeta_p^k
    \le ((\zeta_p N D)^{*j}*G_{\zeta_p}*G_{\zeta_p})(0).
\end{equation}
One contribution to the right-hand side arises from the zero-step walks in each factor
of $G_{\zeta_p}$, and in other contributions there is at least one step from those factors, so
(as in \eqref{e:HDG})
\begin{equation}
\label{e:DGG0}
    \sum_{k=j}^\infty k\pi_k^{(M)} \zeta_p^k
    \le (\zeta_p N D)^{*j}(0) + ((\zeta_p N D)^{*(j+1)}*G_{\zeta_p}*G_{\zeta_p})(0).
\end{equation}
By \eqref{e:Dinf} and the fact that $\zeta_p N \prec 1$,
the first term is of order $N^{-\ceil{j/2}}$.
Via the inverse Fourier transform and the Cauchy--Schwarz inequality,
by \eqref{e:DGbd} the second term is
at most
\begin{align}
    (\zeta_p N)^{j+1}\|D^{*(j+1)}*G_{\zeta_p}*G_{\zeta_p}\|_\infty
    & \prec
     \|\hat D^{j+1}\hat G_{\zeta_p}\|_{\hat 2} \|\hat G_{\zeta_p}\|_{\hat 2}
    \prec
    N^{-(j+1)/2}
    \prec
    N^{-\lceil j/2\rceil},
\label{e:DGG}
\end{align}
and this proves \eqref{e:d-j} for the case $M=1$.

For $M\ge 2$, by \eqref{e:Pizder-bis} and \eqref{e:Gderivbds} we have
\begin{align}
    \sum_{k=2}^\infty k\pi_k^{(M)} z^k
    =
    \|z\partial_z\Pi_z^{(M)}\|_1
    & \le
    (2M-1)
    \|z\partial_z H_z\|_\infty
    \|H_z * G_z \|_\infty^{M-1}
    \nnb & \le (2M-1) \|H_z * G_z\|_\infty^M
    .
\end{align}
The effect of the restriction $k \ge j$ on this bound is to require that
the factors $\|H_z * G_z\|_\infty$ become modified to ensure that at least $j$
steps are taken in total, so there is a restriction that the $i^{\rm th}$ factor
must take at least $j_i$ steps with $j_1+\cdots j_M \ge j$.
As in the case $M=1$, with this restriction the $i^{\rm th}$ factor of $\|H_z*G_z\|_\infty$
can be replaced via an upper bound by $\|D^{*j_i}*G_z*G_z\|_\infty$.
Then, as in \eqref{e:DGG0}--\eqref{e:DGG}, we find that the unrestricted upper bound is replaced by
\begin{equation}
    \sum_{k=j}^\infty k\pi_k^{(M)} \zeta_p^k
    \leq
    C_{M,j} N^{-\sum_{i=1}^M \ceil{j_i/2}}
    \leq C_{M,j} N^{-\lceil j/2 \rceil}
    .
\end{equation}
This completes the proof.
\end{proof}

\subsection{Asymptotic expansion to all orders}
\label{sec:pfexp}

We now prove the part of Theorem~\ref{thm:expansion} concerning existence
of the asymptotic expansions for $\mu_N$ and $A_N$ to all orders with integer coefficients.
The numerical computation of coefficients is deferred to Section~\ref{sec:coefficients}.

\begin{theorem}
\label{thm-expexists}
Let $\lambda_0>0$ be sufficiently small.
Let $m \in \N$, fix $c>0$
(independent of $N$ but possibly depending on $m$), and suppose that $z$ obeys
$\chi_N(z) \in [cN^{m}, \lambda_0 V^{1/2}]$.
Then there are integers
$a_n$ for $n \in \N$, which are universal constants that do not depend on the particular choice
of $z$, such that
\begin{equation}
\label{e:zmainasy}
    z = \sum_{n=1}^m a_n N^{-n} + O(N^{-m-1})
    .
\end{equation}
The constant in the error term depends on $m,\lambda_0,c$, but does not
depend otherwise on $z$.
The same is true for the amplitude $A_N$.
\end{theorem}

\begin{proof}
We first consider the expansion for $z$.
Fix $p \in (0,\frac 12)$.
By Proposition~\ref{prop:p}, since $\chi_N(\zeta_p) \asymp V^p$
 by Proposition~\ref{prop:alphabeta},
it suffices to prove that $\zeta_p$ has an asymptotic
expansion to all orders with integer coefficients.
By \eqref{e:chiPi} we have
\begin{equation}
    \label{e:pcform}
    N\zeta_p
    = 1 - \hat \Pi_{\zeta_p}(0)
    +O(V^{-p}) .
\end{equation}
The last term on the right-hand side is negligible compared to any fixed inverse power
of $N$ so it plays no role.

We follow the basic approach of \cite{HS95} but incorporate the significant
simplifications used in \cite{CLS07}.
To lighten the notation we write $s=N^{-1}$.
We will prove by induction on $m\ge 1$ that there are integers $a_n$  such that
\begin{equation}
\label{e:as.2}
	\zeta_p = \sum_{n=1}^{m }a_n s^n + O(s^{m+1}) .
\end{equation}
The integers $a_n$ will be shown to be universal constants.
The starting point is
\begin{equation}
\label{e:IH}
	\zeta_p
    =
    s[1-\hat\Pi_{\zeta_p}(0)]+O(sV^{-p})
    =
    s \bigg[ 1 - \sum_{M=1}^\infty (-1)^M
    \sum_{i=2}^\infty    \pi_{i}^{(M)} \zeta_p^i  \bigg] +O(sV^{-p}).
\end{equation}
Since $\hat\Pi_{\zeta_p}(0) = O(s)$ (e.g., by \eqref{e:d-N} with $j=1$),
\eqref{e:IH} gives the base case $m=1$ for the induction, with $a_1=1$.

To advance the induction,
we assume now that \eqref{e:as.2} holds for some $m\ge 1$
and we will prove that it holds for $m+1$.
It follows from Proposition~\ref{prop:Piforexpansion}
that for $j \ge 1$,
\begin{equation}
\label{e:d-N-bis}
    \sum_{k=2}^\infty \sum_{M=j}^\infty \pi_k^{(M)} \zeta_p^k \leq C_j s^{j},
\end{equation}
and that for $j \geq 2$ and $M \geq 1$,
\begin{equation}
\label{e:d-j-bis}
    \sum_{k=j}^\infty \pi_k^{(M)} \zeta_p^k \leq C_{M,j} s^{\ceil{j/2}}
    .
\end{equation}
(We do not yet need the
factor $k$ included in the sums
of Proposition~\ref{prop:Piforexpansion} to advance this induction,
that factor is needed only later for $A_N$.
Of course the bounds remain valid without that factor since the left-hand sides
are smaller without it.)
With \eqref{e:d-N-bis}--\eqref{e:d-j-bis}, we see from
\eqref{e:IH} that
\begin{align}
\label{e:bt.st3}
	\zeta_p & =
    s \bigg[ 1 - \sum_{M=1}^{m} (-1)^M
	\sum_{k=2}^{2m} \pi_k^{(M)} \zeta_p^k\bigg]  +O(s^{m+2})
    =
    s \bigg[ 1 -
	\sum_{k=2}^{2m} b_{N,k} \zeta_p^k\bigg]  +O(s^{m+2})
    ,
\end{align}
with
\begin{equation}
\label{e:bNk}
    b_{N,k} = \sum_{M=1}^{m} (-1)^M  \pi_k^{(M)}.
\end{equation}

We classify contributions to $\pi_k^{(M)}$
according to the total
number $\delta$ of dimensions explored by a $k$-step walk $\omega$.
Let
$\pi_{k,\delta}^{(M)}$ denote the contribution to $\pi_k^{(M)}$ due to walks
starting at $0$ which
explore exactly $\delta$ dimensions, with the first step taken
with the first coordinate,
the first subsequent step involving a new coordinate
taken with the second coordinate,
the first subsequent step with a new coordinate taken with the third coordinate, and so on.
Then, because the number of dimensions cannot exceed $k-1$ (because the last step
must close a loop), it follows by symmetry that
\begin{equation}
\label{e:pidelta}
    \pi_k^{(M)} = \sum_{\delta =1}^{k-1}   \pi_{k,\delta}^{(M)}
     \prod_{j=0}^{\delta-1} (N-j).
\end{equation}
By definition, $\pi_{k,\delta}^{(M)}$ is a nonnegative integer which is independent
of $N$, a universal number which counts certain lace diagrams.
With \eqref{e:bNk}, we find that
\begin{equation}
\label{e:coeff.2}
    b_{N,k} = \sum_{q = 1}^{k-1} \beta_{q,k} N^q
\end{equation}
with integer coefficients $\beta_{q,k}$ which are independent of $N$.

By the induction hypothesis \eqref{e:as.2},
\begin{equation}
	\zeta_p^k =
	\bigg[ \sum_{n=1}^{m} a_n s^n + O(s^{m+1})
	\bigg]^k
	= s^k \bigg[ \sum_{n=0}^{m-1} \gamma_{n,k} s^n + O(s^{m})
	\bigg]
\end{equation}
with $N$-independent integer coefficients $\gamma_{n,k}$.
With \eqref{e:bt.st3} and \eqref{e:coeff.2}, this gives
\begin{align}
	\zeta_p & =
    s
    \Bigg( 1-
	\sum_{k=2}^{2m} \sum_{q = 1}^{k-1} \beta_{q,k} s^{-q} s^k
    \bigg[ \sum_{n=0}^{m-1} \gamma_{n,k} s^n + O(s^{m}) \bigg] \Bigg)
    +O(s^{m+2})
    .
\end{align}
The only term in the above product which can give rise to an non-integer
power or coefficient is the $O(s^{m})$ term.  However this term
is multiplied by $s^{1 -q+k} \leq  s^2$, and hence gives rise to a
contribution which is $O(s^{m+2})$.  Therefore
\begin{equation}
\label{e:Pi.con}
	\zeta_p
	= \sum_{n=1}^{m+1} d_{n} s^n + O(s^{m+2}),
\end{equation}
with integer coefficients $d_{n}$,
which must agree with $a_n$
for $n \leq m$.  This gives
\eqref{e:as.2} with $m$ replaced by $m+1$, so the induction is advanced
and the proof of \eqref{e:zmainasy} is complete.

Recall that the amplitude
\begin{equation}
    A_N
    =
    \frac{1}{F_N(\zeta_p) - \zeta_pF_N'(\zeta_p)}
    =
    \frac{1}{F_N(\zeta_p)+ \zeta_p N + \zeta_p\partial_z\hat\Pi_{\zeta_p}(0)}
\end{equation}
was defined in the proof of Corollary~\ref{cor:varphin}.
The existence of the expansion for $A_N$
then follows similarly from \eqref{e:d-N}--\eqref{e:d-j}
(now we do need their factor $k$), by substitution of the expansion for $\zeta_p$
into
\begin{eqnarray}
\label{e:Ad}
    \frac{1}{A_N} = \zeta_p N +
    \sum_{k=2}^{2m} \sum_{M=1}^m (-1)^M k\pi_k^{(M)} \zeta_p^k
    +O(N^{-m-1})
    .
\end{eqnarray}
Again the coefficients in the expansion are universal integers.
\end{proof}

\subsection{Coefficient computation}
\label{sec:coefficients}

The expansion coefficients for $z$ in \eqref{e:zmainasy}
(equivalently, for $\zeta_p$) can be computed
from \eqref{e:bt.st3} and \eqref{e:pidelta} once suitably many of
the lace graph counts $\pi_{k,\delta}^{(M)}$
are known.
To compute to within an error of order $s^6$, only $M \le 4$
and $k \le 8$ are needed, by the first equality of \eqref{e:bt.st3}.
Also, we can restrict to $k-\delta \le 4$ because
in a term $\pi_{k,\delta}^{(M)}\zeta_p^k$ the largest possible contribution
is of order
$N^\delta s^k=s^{k-\delta}$ and if $k-\delta \ge 5$ then this contributes to $\zeta_p$ only at order $s^6$, due to the prefactor $s$ in \eqref{e:bt.st3}.

To discuss the
enumeration of lace graphs, we elaborate on the definition of an $M$-loop diagram
as follows.  Recall the definition of compatible edges from Definition~\ref{def:lace_presc}.
If the $2M-1$ subwalks in the $M$-loop diagram of Figure~\ref{fig:lacetop} are  labelled
in the order they occur in a walk as
$1,2,\ldots, 2M-1$, then the subwalks are mutually avoiding (apart from the
required intersections) due to the compatible edges, with the following patterns:
$[123]$ for $M=2$;
$[1234], [345]$ for $M=3$;
$[1234], [3456], [567]$ for $M=4$.
This means, e.g., for $M=4$, that subwalks $1,2,3,4$
are mutually avoiding apart from the enforced intersections explicitly
depicted in Figure~\ref{fig:lacetop},
as are subwalks $3,4,5,6$ and subwalks $5,6,7$.  However, subwalks not
grouped together are permitted to intersect, e.g., for $M=4$,
subwalks $1,2$ are permitted
to intersect subwalks $5,6,7$, and subwalks $3$ and $4$ can intersect subwalk $7$.

Extensive computer assisted
enumerations of lace graphs for $\Z^d$ are given in \cite{CLS06arXiv}.
Lace graphs on the hypercube are a subset of those on $\Z^d$.
The relevant nonzero counts
are found to be:
\begin{align}
    \pi_{2,1}^{(1)} & = 1,
    \qquad
    \pi_{4,2}^{(1)}  = 1,
    \qquad
    \pi_{6,3}^{(1)}  = 4,
    \qquad
    \pi_{8,4}^{(1)}  = 27,
\end{align}
\begin{align}
    \pi_{3,1}^{(2)}  = 1,
    \qquad
    \pi_{5,2}^{(2)}  = 3,
    \qquad
    \pi_{7,3}^{(2)}  = 15
    ,
    \qquad
    \pi_{4,1}^{(3)}  = 1,
    \qquad
    \pi_{6,2}^{(3)}  = 5,
    \qquad
    \pi_{5,1}^{(4)}  = 1.
\end{align}
These numbers arise by inspection from
Tables~18, 20, 22, 24 of \cite{CLS06arXiv}
by discounting configurations which
are possible on $\Z^d$ but not on $\Q^N$,
as well as from the trivial counts $\pi_{k,1}^{(M)}=\delta_{k,M+1}$.
E.g., $\pi_{6,2}^{(1)}=3$ on $\Z^d$ but it is zero on
$\Q^N$ where the 2-dimensional subspace contains only $4$ vertices.
With computer assistance, it would be possible to enumerate more lace graphs on $\Q^N$ and
thereby compute more terms in the asymptotic expansion for $z_N$.

To compute the expansion coefficients for $z_N=z_N(\lambda_0)$ (recall
Proposition~\ref{prop:p}), as in \eqref{e:bt.st3} we use
\begin{align}
\label{e:iteration1}
	z_N & =
    s \left [ 1 - \sum_{M=1}^{m} (-1)^M
	\sum_{k=2}^{2m} \pi_k^{(M)} z_N^k\right]  +O(s^{m+2})
\end{align}
iteratively with
\begin{equation}
\label{e:iteration2}
    \pi_k^{(M)} = \sum_{\delta =1}^{k-1}   \pi_{k,\delta}^{(M)}
     \prod_{j=0}^{\delta-1} (N-j)
\end{equation}
as follows.
We insert $z_N = s + O(s^2)$ into \eqref{e:iteration1} with $m=1$ and obtain
\begin{align}
    z_N & =
        s \left [ 1 - (-1)^1 \pi_2^{(1)}z_N^2 \right]  +O(s^{3})
        \nnb & =
        s\left[1 + N \cdot 1 \cdot s^2 \right] +O(s^{3})
        =
        s + s^2 + O(s^3).
\end{align}
Another iteration gives
\begin{align}
    z_N & =
        s \left [ 1 - (-1)^1 (\pi_2^{(1)}z_N^2 + \pi_4^{(1)}z_N^4)
        - (-1)^2  \pi_3^{(2)} z_N^3 \right]  +O(s^{4})
        \nnb & =
        s\left[1 + (N \cdot 1 \cdot (s+s^2)^2 + N(N-1)\cdot 1 \cdot s^4)
        - N \cdot 1 \cdot s^3  \right] +O(s^{4})
        \nnb &
        = s + s^2 + 2s^3 + O(s^4).
\end{align}
Continuing in this way, we find that
\begin{align}
    z_N & = s + s^2 + 2s^3 + 7s^4 + 39 s^5 + O(s^6).
\end{align}

Now we can simply substitute the expansion for $z_N$ into \eqref{e:Ad}
to get the
expansion for $A_N^{-1}$ (and hence for $A_N$)
up to and including terms of order $s^{4}$.  The result is
\begin{equation}
    A_N = 1 + s + 4s^2 + 26s^3 + 231s^4 +O(s^5).
\end{equation}

\section*{Acknowledgements}

I am grateful to Emmanuel Michta for his influence on this paper through
our collaboration on \cite{MS22} and for helpful comments on a preliminary
version, to Yucheng Liu for
assistance with the numerical computation of coefficients in the expansions for
$\mu_N$ and $A_N$ reported in Section~\ref{sec:coefficients}, and to Yuliang Shi
for comments on a preliminary version.
This work was supported in part by NSERC of Canada.



\end{document}